\documentclass[11pt,a4paper]{article}
\usepackage{amsmath}
\usepackage{amsfonts}
\usepackage{eurosym}
\usepackage{geometry}
\usepackage{color}

\newtheorem{theorem}{Theorem}[section]

\newtheorem{corollary}[theorem]{Corollary}

\newtheorem{lemma}[theorem]{Lemma}

\newtheorem{proposition}[theorem]{Proposition}
\newtheorem{remark}[theorem]{Remark}

\newenvironment{proof}[1][Proof]{\noindent\textbf{#1.} }{\ \rule{0.5em}{0.5em}}

\begin{document}

\title{Partial balayage for the Helmholtz equation}
\date{}
\author{Stephen J. Gardiner and Tomas Sj\"{o}din}
\maketitle

\begin{abstract}
Kow, Larson, Salo and Shahgholian recently initiated the study of quadrature
domains for the Helmholtz equation and developed an associated theory of
partial balayage of measures. The present paper offers an alternative
approach to partial balayage in this context that yields stronger results.
Applications are given to quadrature domains and to a domain evolution
question that is analogous to Hele-Shaw flow.
\end{abstract}

\section{Introduction\label{Intro}\protect\footnotetext{%
\noindent 2020 Mathematics Subject Classification: 35J05 31B15. \newline
Keywords: Helmholtz equation, metaharmonic function, quadrature domain,
domain evolution}}

Let $\mu $ be a (positive) measure with compact support in Euclidean space $%
\mathbb{R}^{N}$ ($N\geq 2$) and $U^{\mu }$ be its Newtonian (or logarithmic,
if $N=2$) potential. Given a bounded open subset $\Omega $ containing $%
\mathrm{supp}(\mu )$, classical balayage sweeps $\mu $ out of $\Omega $ to
yield a measure $\mu _{1}$ on $\partial \Omega $ such that $U^{\mu
_{1}}<U^{\mu }$ in the components of $\Omega $ that are charged by $\mu $
and $U^{\mu _{1}}=U^{\mu }$ elsewhere (apart, possibly, from a polar subset
of $\partial \Omega $). The notion of partial balayage was introduced more
recently by Gustafsson and Sakai and expounded by them in \cite{GuSa}. (An
alternative account may be found in \cite{GaSj}.) Provided $\mu $ is
sufficiently concentrated it yields a bounded open set $\Omega $ containing $%
\mathrm{supp}(\mu )$ such that $U^{m|_{\Omega }}<U^{\mu }$ in $\Omega $ and $%
U^{m|_{\Omega }}=U^{\mu }$ elsewhere, where $m$ denotes Lebesgue measure on $%
\mathbb{R}^{N}$. The main differences from classical balayage are that the
measure $\mu $ is not completely swept out of $\Omega $ and that it is the
set $\Omega $, rather than the measure $\mu _{1}$, that is determined by the
sweeping process. The resulting open set $\Omega $ can be shown to satisfy%
\begin{equation}
\int sd\mu \leq \int_{\Omega }sdm\text{ \ \ \ for any integrable subharmonic
function }s\text{ on }\Omega ;  \label{sqdh}
\end{equation}%
that is, $\Omega $ is a quadrature domain with respect to $\mu $ for
subharmonic functions on $\Omega $. (Of course, equality will then hold in (%
\ref{sqdh}) for integrable harmonic functions on $\Omega $.) Further, by
considering a family of measures of the form $\mu _{t}=m|_{\Omega }+t\delta
_{z}$ ($t>0$), where $\Omega $ is a domain and $\delta _{z}$ is the unit
mass concentrated at a point $z$ of $\Omega $, partial balayage can be used
to study Hele-Shaw flow and related Laplacian growth processes.

Recently Kow et al. \cite{KLSS}, motivated by questions about inverse
scattering theory, investigated quadrature domains for solutions of the
Helmholtz equation 
\begin{equation}
(\Delta +k^{2})h=0,  \label{Hel}
\end{equation}%
where $k>0$, and developed a corresponding theory of partial balayage. The
purpose of the present paper is to offer an alternative approach to partial
balayage for this equation that significantly strengthens several aspects of
the theory and has applications to quadrature domains and domain evolution
problems in this context.

Solutions of (\ref{Hel}) will be called $k$-\textit{metaharmonic functions}.
More generally, any upper semicontinuous function $s$ which satisfies $%
(\Delta +k^{2})s\geq 0$ in the sense of distributions will be called $k$-%
\textit{metasubharmonic}, and a function $u$ will be called $k$-\textit{%
metasuperharmonic} if $-u$ is $k$-metasubharmonic. We will say that the $k$-%
\textit{maximum principle} holds on a bounded open set $\Omega $ if any
upper bounded $k$-metasubharmonic function $s$ on $\Omega $ which satisfies 
\begin{equation*}
\lim \sup_{x\rightarrow z}s(x)\leq 0\text{ \ for all }z\text{ in }\partial
\Omega \text{ apart, possibly, from a polar set}
\end{equation*}%
must also satisfy $s\leq 0$ on $\Omega $. (It is important that we use the
value $0$ in this definition.) Theorem 1.1 in \cite{BNV}, which concerns
more general elliptic operators, shows that the $k$-maximum principle holds
on $\Omega $ if and only if $k^{2}<\lambda _{1}(\Omega )$, where%
\begin{equation}
\lambda _{1}(\Omega )=\max \left\{ \lambda >0:\text{there is a positive }%
\sqrt{\lambda }\text{-metasuperharmonic function on }\Omega \right\} .
\label{lambda1}
\end{equation}%
In Proposition \ref{MP} below we will provide an alternative proof of this
equivalence that uses only tools from classical potential theory. It will be
sufficient to do this for bounded domains $\Omega $, since the result for
bounded open sets then follows by considering each component separately.
Propositions \ref{equiv} and \ref{evnoncon1} will explain this point and
also why $\lambda _{1}$ agrees with the usual formula for the first
eigenvalue of the Laplacian when $\Omega $ is connected.

Radial solutions of (\ref{Hel}) on $\mathbb{R}^{N}\backslash \{0\}$\ are of
the form $h(x)=|x|^{-\alpha }f(k|x|)$, where 
\begin{equation*}
\alpha =(N-2)/2
\end{equation*}%
and $f$ satisfies Bessel's equation of order $\alpha $, namely 
\begin{equation*}
t^{2}f^{\prime \prime }(t)+tf^{\prime }(t)+\left( t^{2}-\alpha ^{2}\right)
f(t)=0\text{ \ \ \ }(t>0).
\end{equation*}%
Solutions to this last equation have the form $f(t)=aJ_{\alpha }(t)+bY_{\alpha
}(t)$, where $J_{\alpha }$ and $Y_{\alpha }$ are the Bessel functions of
order $\alpha $ of the first and second kind respectively (see Watson \cite%
{Wat}). Hence radial solutions of (\ref{Hel}) have the representation%
\begin{equation}
h(x)=|x|^{-\alpha }\left( aJ_{\alpha }(k|x|)+bY_{\alpha }(k|x|)\right) .
\label{radial}
\end{equation}%
If $b=-k^{\alpha }/(2^{\alpha +2}\pi ^{\alpha })$, then the above function
satisfies $-(\Delta +k^{2})h=\delta _{0}$ in the sense of distributions. In
particular, this is true of the function 
\begin{equation*}
\Psi _{k}(x)=b|x|^{-\alpha }Y_{\alpha }(k|x|),
\end{equation*}%
and we define the $k$\textit{-potential} of a compactly supported measure $%
\mu $ by 
\begin{equation*}
U_{k}^{\mu }(x)=\int \Psi _{k}(x-y)d\mu (y).
\end{equation*}

Next we define the collection 
\begin{equation*}
\mathcal{F}_{k}(\mu )=\left\{ v\in \mathcal{D}^{\prime }({\mathbb{R}}%
^{N}):-(\Delta +k^{2})v\leq 1,\text{ \ }v\leq U_{k}^{\mu }\text{ and }%
\{v<U_{k}^{\mu }\}\text{ is bounded}\right\} .
\end{equation*}%
If $\mathcal{F}_{k}(\mu )\neq \emptyset $, then (as we will see later)
standard arguments show that $\mathcal{F}_{k}(\mu )$ has a largest element,
which has an upper semicontinuous representative. We denote this function by 
$V_{k}^{\mu }$ and define the open set $\omega _{k}(\mu )=\{V_{k}^{\mu
}<U_{k}^{\mu }\}$. The measure $\mathcal{B}_{k}(\mu )=-(\Delta
+k^{2})V_{k}^{\mu }$ is called the $k$\textit{-partial balayage of }$\mu $%
\textit{\ onto }$m$. Its structure is described in the following result.

\begin{theorem}
\label{Str0}If $\mathcal{F}_{k}(\mu )\neq \emptyset $, then $\mathcal{B}%
_{k}(\mu )=m|_{\omega _{k}(\mu )}+\mu |_{\omega _{k}(\mu )^{c}}\leq m$.
\end{theorem}

If $\mu \leq m$, then trivially $V_{k}^{\mu }=U_{k}^{\mu }$ and $\omega
_{k}(\mu )=\emptyset $. The next result establishes a lower bound for $%
\lambda _{1}(\omega _{k}(\mu ))$ when $\omega _{k}(\mu )\neq \emptyset $.

\begin{theorem}
\label{mpsat0}Let $\mu $ be a compactly supported measure such that $%
\mathcal{F}_{k}(\mu )\neq \emptyset $. If $\omega _{k}(\mu )\neq \emptyset $%
, then $\lambda _{1}(\omega _{k}(\mu ))\geq k^{2}$.
\end{theorem}

Let $B_{r}(z)$ denote the open ball in $\mathbb{R}^{N}$ of centre $z$ and
radius $r$, and let $B_{r}=B_{r}(0)$. We note that $\lambda
_{1}(B_{r})=(j_{\alpha ,1}/r)^{2}$, where $j_{\nu ,1}$ denotes the first
positive zero of $J_{\nu }$. The $k$-maximum principle thus holds on $B_{r}$
if and only if $r<R_{k}$, where 
\begin{equation*}
R_{k}=j_{\alpha ,1}/k.
\end{equation*}%
The quantity 
\begin{equation}
c_{k}(r)=\left( 2\pi r/k\right) ^{N/2}J_{N/2}(kr)  \label{ck}
\end{equation}%
is also significant for us, because 
\begin{equation*}
c_{k}(r)h(z)=\int_{B_{r}(z)}h\,dm
\end{equation*}%
when $h$ is a $k$-metaharmonic function on a neighbourhood of $\overline{B}%
_{r}(z)$ (see Proposition \ref{MVP} below). Unlike the case of harmonic
functions, $c_{k}(\cdot )$ has variable sign. The simplest case of $k$%
-partial balayage, where $\mu $ is of the form $c\delta _{z}$, is described
in the next result.

\begin{proposition}
\label{pmb0}Let $c>0$. \newline
(a) Then $\mathcal{F}_{k}(c\delta _{z})\neq \emptyset $ if and only if $%
0<c\leq c_{k}(R_{k})$.\newline
(b) If $0<c\leq c_{k}(R_{k})$, then $\mathcal{B}_{k}(c\delta
_{z})=m|_{B_{r}(z)}$, where $r$ is the unique value in $[0,R_{k}]$ such that 
$c=c_{k}(r)$.
\end{proposition}

More generally, if there exists $r \in (0,R_k]$ such that $\mu(B_r(x)) \le c_k(r)$ for all $x$, then $\mathcal{F}_k(\mu) \ne \emptyset$ (see Proposition \ref{mollexist} below, which is due to Simon Larson).

We say that a bounded open set $\Omega $ is a \textit{quadrature domain with
respect to }$\mu $ \textit{for }$k$-\textit{metasubharmonic functions} if $%
\mathrm{supp}(\mu )\subset \Omega $ and%
\begin{equation}
\int_{\Omega }sdm\geq \int sd\mu \ \ \ \text{for any integrable }k\text{%
-metasubharmonic function }s\text{ on }\Omega .  \label{sqd}
\end{equation}%
Our results on $k$-partial balayage are applicable to quadrature domains
because of the following theorem, which follows immediately from Proposition
7.5 in \cite{KLSS}. Its extra generality, in comparison with Theorem 1.6 of
that paper, arises from our new results about when $\mathcal{F}_{k}(\mu
)\neq \emptyset $.

\begin{theorem}
\label{qd0}Suppose that $\mathcal{F}_{k}(\mu )\neq \emptyset $ and $\mu
(\omega _{k}(\mu )^{c})=0$. Then (\ref{sqd}) holds when $\Omega =\omega
_{k}(\mu )$. In particular, if $\mathrm{supp}(\mu) \subset \Omega $, then $%
\Omega $ is a quadrature domain with respect to $\mu $ for $k$%
-metasubharmonic functions.
\end{theorem}

The above hypothesis that $\mu (\omega _{k}(\mu )^{c})=0$ can be achieved by
modifying $\mu $ since, if $\omega _{k}(\mu )\neq \emptyset $ and we define $%
\mu _{1}=\mu |_{\omega _{k}(\mu )}$, then Theorem \ref{Str0} shows that $%
\omega _{k}(\mu _{1})=\omega _{k}(\mu )$. Also, the condition that $\mathrm{%
supp}(\mu) \subset \Omega $ will hold if, for example, $U_{k}^{\mu }=\infty $
on $\mathrm{supp}(\mu)$.

We can also apply our theory of partial balayage to obtain the following
domain evolution result inspired by Hele-Shaw flow. We will see later that,
if $\Omega $ is a bounded domain such that $\lambda _{1}(\Omega )>k^{2}$ and 
$z\in \Omega $, then there is a unique measure $\nu _{k,z}^{\Omega }$
supported on $\partial \Omega $ such that $U_{k}^{\delta _{z}}-U_{k}^{\nu
_{k,z}^{\Omega }}\geq 0$ with equality on $\Omega ^{c}$ apart, possibly,
from a polar subset of $\partial \Omega $.

\begin{theorem}
\label{HS0}Let $\Omega $ and $D$ be bounded domains such that $\overline{%
\Omega }\subset D$ and $\lambda _{1}(D)\geq k^{2}$, and let $\mu
_{t}=t\delta _{z}+m|_{\Omega }$ $(t>0)$, where $z\in \Omega $. Then the set 
\begin{equation*}
I=\{t>0:\mathcal{F}_{k}(\mu _{t})\neq \emptyset \}
\end{equation*}%
is a bounded interval of the form $(0,T]$. Further, \newline
(a) $\Omega ${$\subset \omega _{k}(\mu _{t})$ for all $t\in (0,T]$; \newline
(b) $\omega _{k}(\mu _{t})\nearrow \omega _{k}(\mu _{s})$ as $t\nearrow s$
for all $s\leq T$; \newline
(c) the following Hele-Shaw law holds: 
\begin{equation*}
\omega _{k}(\mu _{t+\varepsilon })=\omega _{k}\left( m|_{\omega _{k}(\mu
_{t})}+\varepsilon \nu _{k,z}^{\omega _{k}(\mu _{t})}\right) \text{\ \ }%
(0<t<t+\varepsilon \leq T).
\end{equation*}%
}
\end{theorem}

In the next two sections we will assemble material about the potential
theory of the Helmholtz equation and then develop the theory of partial
balayage for this equation. We note that the paper \cite{KLSS} treated
partial balayage in the case where $\mu \in L^{\infty }$ (whence $\mu $ has
variable sign) and $\mu $ is supported in a small ball. Below we will treat
partial balayage onto a measure of the form $\rho m$, where $\rho \in
L^{\infty }({\mathbb{R}}^{N})$ and $c^{-1}\leq \rho \leq c$, for some $c>0$,
and will point out how this allows us also to cover the case where $\mu \in
L^{\infty }$. More importantly, our results establish the existence of the $%
k $-partial balayage of a much larger collection of measures $\mu $.

Theorems \ref{Str0} and \ref{mpsat0} will be proved in
Section \ref{PB}. Theorem \ref{qd0} is a special case of Theorem \ref{qd2},
and Theorem \ref{HS0} is a consequence of Theorem \ref{HS2} and Remark \ref%
{closcons}. Proposition \ref{pmb0} will be proved in Section \ref%
{constantrho} along with other explicit computations of $k$-partial balayage.

\section{Potential theory for the Helmholtz equation}

We collect below some useful properties of Bessel functions.

\begin{lemma}
\label{Bessel}(i) $\dfrac{d}{dt}\left( t^{\nu }J_{\nu }(t)\right) =t^{\nu
}J_{\nu -1}(t)$ and $\dfrac{d}{dt}\left( t^{-\nu }J_{\nu }(t)\right)
=-t^{-\nu }J_{\nu +1}(t)$;\newline
(ii) $\dfrac{d}{dt}\left( t^{\nu }Y_{\nu }(t)\right) =t^{\nu }Y_{\nu -1}(t)$
and $\dfrac{d}{dt}\left( t^{-\nu }Y_{\nu }(t)\right) =-t^{-\nu }Y_{\nu
+1}(t) $;\newline
(iii) $J_{\nu }(t)Y_{\nu }^{\prime }(t)-Y_{\nu }(t)J_{\nu }^{\prime
}(t)=2/(\pi t)$ when $t>0$;\newline
(iv) $J_{\nu }(t)Y_{\nu +1}(t)-Y_{\nu }(t)J_{\nu +1}(t)=-2/(\pi t)$ when $%
t>0 $;\newline
(v) $\Gamma (N/2)2^{\alpha }t^{-\alpha }J_{\alpha }(t)\rightarrow 1$ as $%
t\rightarrow 0+$;\newline
(vi) $2^{-\alpha }\pi t^{\alpha }Y_{\alpha }(t)\rightarrow -\Gamma (\alpha )$%
\ if $\alpha >0,$ and $(\pi /2)Y_{0}(t)/\log t\rightarrow 1$, as $%
t\rightarrow 0+$.
\end{lemma}

\begin{proof}
These properties may all be found in \cite{Wat}: we refer to p.45, p.66,
p.76(1) and p.77(12) for parts (i) - (iv), respectively, to p.40(8) for part
(v), and to p.64(2), p.62(3) and p.60(2) for part (vi).
\end{proof}

\bigskip

We define 
\begin{equation}
d_{k}(r)=(2\pi r)^{N/2}\frac{J_{\alpha }(kr)}{k^{\alpha }}  \label{dk}
\end{equation}%
and note from (\ref{ck}) and Lemma \ref{Bessel}(i) that 
\begin{equation}
c_{k}(r)=\int_{0}^{r}d_{k}(t)\,dt.  \label{ckdk}
\end{equation}%
The quantity $d_{k}$, like $c_{k}$, has variable sign, but both are positive
on the interval $(0,R_{k})$. The first part of the following result is
classical (cf. pp.288,289 of \cite{CH}), but we provide a proof so that we
can include also the case of $k$-metasubharmonic functions. We denote
surface area measure on a sphere by $\sigma $.

\begin{proposition}
\label{MVP} If $h$ is $k$-metaharmonic on $B_{R}(z)$ and $0<r<R$, then 
\begin{equation}
d_{k}(r)h(z)=\int_{\partial B_{r}(z)}h\,d\sigma \text{ \ \ and \ \ }%
c_{k}(r)h(z)=\int_{B_{r}(z)}h\,dm.  \label{MVPi}
\end{equation}%
Further, if $w$ is $k$-metasubharmonic on $B_{R}(z)$ and $0<r<R\leq R_{k}$,
then 
\begin{equation}
w(z)\leq \frac{1}{c_{k}(r)}\int_{B_{r}(z)}w\,dm\leq \frac{1}{d_{k}(r)}%
\int_{\partial B_{r}(z)}w\,d\sigma ,  \label{MVPi2}
\end{equation}%
and these last two expressions are increasing as functions of $r$ on $(0,R)$%
. The first inequality in (\ref{MVPi2}) remains valid when $r=R_{k}<R$.
\end{proposition}

\begin{proof}
Without loss of generality we may assume that $z=0$. Let $h$ be $k$%
-metaharmonic on $B_{R}$. If $d_{k}(r)\neq 0$, then the function 
\begin{equation*}
\tilde{\Phi}_{k,r}(x)=\frac{k^{\alpha }\left\vert x\right\vert ^{-\alpha }}{%
2^{\alpha +2}\pi ^{\alpha }}\left\{ \frac{Y_{\alpha }(kr)}{J_{\alpha }(kr)}%
J_{\alpha }(k\left\vert x\right\vert )-Y_{\alpha }(k\left\vert x\right\vert
)\right\}
\end{equation*}%
is a radial fundamental solution to the Helmholtz equation that vanishes on $%
\partial B_{r}$. Further, denoting by $n$ the outward unit normal to a
sphere, we note from Lemma \ref{Bessel} that%
\begin{equation}
\nabla \tilde{\Phi}_{k,r}(x)\cdot \frac{x}{\left\vert x\right\vert }=\frac{%
k^{\alpha }\left\vert x\right\vert ^{-\alpha }(-k)}{2^{\alpha +2}\pi
^{\alpha }J_{\alpha }(kr)}\left\{ Y_{\alpha }(kr)J_{\alpha +1}(k\left\vert
x\right\vert )-J_{\alpha }(kr)Y_{\alpha +1}(k\left\vert x\right\vert
)\right\} \text{ \ \ \ }(x\neq 0),  \label{grad0}
\end{equation}%
whence 
\begin{equation}
\nabla \tilde{\Phi}_{k,r}(x)\cdot n\left\vert _{\partial B_{r}}\right. =-%
\frac{1}{d_{k}(r)}\text{ \ \ and \ \ }\sigma (\partial B_{\varepsilon
})\nabla \tilde{\Phi}_{k,r}(x)\cdot n\left\vert _{\partial B_{\varepsilon
}}\right. \rightarrow -1\text{ \ \ }(\varepsilon \rightarrow 0+).
\label{grad}
\end{equation}%
Thus, if $0<\varepsilon <r$,%
\begin{eqnarray*}
0 &=&\int_{B_{r}\backslash \overline{B}_{\varepsilon }}(-(\Delta +k^{2})%
\tilde{\Phi}_{k,r})h\,dm=\int_{B_{r}\backslash \overline{B}_{\varepsilon
}}\left( \tilde{\Phi}_{k,r}\Delta h-h\Delta \tilde{\Phi}_{k,r}\right) \,dm \\
&=&\int_{\partial B_{\varepsilon }}\left\{ h(\nabla \tilde{\Phi}_{k,r}\cdot
n)-(\nabla h\cdot n)\tilde{\Phi}_{k,r}\right\} d\sigma -\int_{\partial
B_{r}}h(\nabla \tilde{\Phi}_{k,r}\cdot n)d\sigma \\
&\rightarrow &-h(0)+\frac{1}{d_{k}(r)}\int_{\partial B_{r}}hd\sigma \text{ \
\ \ }(\varepsilon \rightarrow 0+).
\end{eqnarray*}%
This establishes the first identity in (\ref{MVPi}) when $d_{k}(r)\neq 0$.
It extends to the case where $d_{k}(r)=0$ on taking limits. The second
identity then follows, by (\ref{ckdk}).

Now suppose that $0<r<R\leq R_{k}$ and that $w$ is metasubharmonic on $B_{R}$%
. We may assume, by the use of a smoothing kernel, that $w$ is smooth on a
neighbourhood of $\overline{B}_{r}$, and we note from (\ref{grad}) that $%
\tilde{\Phi}_{k,r}\geq 0$ in $B_{r}$. Since also $\Delta w\geq -k^{2}w$, we
can argue as before to see that 
\begin{eqnarray}
0 &=&\int_{B_{r}\backslash \overline{B}_{\varepsilon }}(-(\Delta +k^{2})%
\tilde{\Phi}_{k,r})w\,dm\leq \int_{B_{r}\backslash \overline{B}_{\varepsilon
}}\left( \tilde{\Phi}_{k,r}\Delta w-w\Delta \tilde{\Phi}_{k,r}\right) \,dm 
\notag \\
&=&\int_{\partial B_{\varepsilon }}\left\{ w(\nabla \tilde{\Phi}_{k,r}\cdot
n)-(\nabla w\cdot n)\tilde{\Phi}_{k,r}\right\} d\sigma -\int_{\partial
B_{r}}w(\nabla \tilde{\Phi}_{k,r}\cdot n)d\sigma  \notag \\
&\rightarrow &-w(0)+\frac{1}{d_{k}(r)}\int_{\partial B_{r}}wd\sigma \text{ \
\ \ }(\varepsilon \rightarrow 0+).  \label{w00}
\end{eqnarray}%
The first inequality in (\ref{MVPi2}) now follows, by (\ref{ckdk}) (and
remains valid even when $r=R_{k}<R$). We next note from (\ref{w00}) that 
\begin{equation*}
\frac{1}{d_{k}(r)}\int_{\partial B_{r}}wd\sigma =w(0)+\int_{B_{r}}\tilde{\Phi%
}_{k,r}(\Delta +k^{2})wdm.
\end{equation*}%
%
%
%
%
%
Since, by Lemma \ref{Bessel}(iii), the quantity $Y_{\alpha }(ks)/J_{\alpha
}(ks)$ is increasing with $s$ on $(0,R_{k})$, it follows that, if $%
0<s<r<R_{k}$, then $0\leq \tilde{\Phi}_{k,s}\leq \tilde{\Phi}_{k,r}$ in $%
B_{s}$. Hence 
\begin{equation*}
\int_{B_{r}}\tilde{\Phi}_{k,r}(\Delta +k^{2})wdm\geq \int_{B_{s}}\tilde{\Phi}%
_{k,r}(\Delta +k^{2})wdm\geq \int_{B_{s}}\tilde{\Phi}_{k,s}(\Delta
+k^{2})wdm,
\end{equation*}%
so $\left( d_{k}(r)\right) ^{-1}\int_{\partial B_{r}}wd\sigma $ increases
with $r$ on $(0,R_{k})$. Further,%
\begin{align*}
\frac{1}{c_{k}(r)}\int_{B_{r}}wdm& =\frac{1}{c_{k}(r)}\int_{0}^{r}\left(
d_{k}(t)\left( \frac{1}{d_{k}(t)}\int_{\partial B_{t}}w\,d\sigma \right)
\right) \,dt \\
& \leq \frac{1}{c_{k}(r)}\int_{0}^{r}\left( d_{k}(t)\left( \frac{1}{d_{k}(r)}%
\int_{\partial B_{r}}w\,d\sigma \right) \right) \,dt=\frac{1}{d_{k}(r)}%
\int_{\partial B_{r}}w\,d\sigma ,
\end{align*}%
which completes the proof of (\ref{MVPi2}). Finally,%
\begin{align*}
\frac{d}{dr}\left( \frac{1}{c_{k}(r)}\int_{B_{r}}w\,dm\right) & =\frac{1}{%
c_{k}(r)}\int_{\partial B_{r}}wd\sigma -\frac{d_{k}(r)}{c_{k}(r)^{2}}%
\int_{B_{r}}w\,dm \\
& =\frac{d_{k}(r)}{c_{k}(r)}\left\{ \frac{1}{d_{k}(r)}\int_{\partial
B_{r}}wd\sigma -\frac{1}{c_{k}(r)}\int_{B_{r}}w\,dm\right\} \geq 0,
\end{align*}%
so $\left( c_{k}(r)\right) ^{-1}\int_{B_{r}}w\,dm$ also increases with $r$
on $(0,R)$.
\end{proof}

\bigskip

\begin{corollary}
Let $r=|x|$. Then 
\begin{equation}
U_{k}^{\sigma |_{\partial B_{t}}}(x)=\left\{ 
\begin{array}{lc}
b_{k}(t)r^{-\alpha }J_{\alpha }(kr) & (r\leq t) \\ 
&  \\ 
-\frac{\pi t^{N/2}J_{\alpha }(kt)}{2}r^{-\alpha }Y_{\alpha }(kr) & (r>t)%
\end{array}%
\right. ,  \label{C1}
\end{equation}%
where $b_{k}(t)=-\pi t^{N/2}Y_{\alpha }(kt)/2$. Also 
\begin{equation}
U_{k}^{m|_{B_{t}}}(x)=\left\{ 
\begin{array}{lc}
a_{k}(t)r^{-\alpha }J_{\alpha }(kr)-1/k^{2} & (r\leq t) \\ 
&  \\ 
-\frac{\pi t^{N/2}J_{N/2}(kt)}{2k}r^{-\alpha }Y_{\alpha }(kr) & (r>t)%
\end{array}%
\right. ,  \label{C2}
\end{equation}%
where 
\begin{equation*}
a_{k}(t)=\frac{-\pi t^{N/2}}{2k}Y_{N/2}(kt).
\end{equation*}
\end{corollary}

\begin{proof}
We obtain the second formula in (\ref{C1}) by applying the first formula in (%
\ref{MVPi}) to the function $h=\Psi _{k}(x-\cdot )$ and using (\ref{dk}).
The first formula in (\ref{C1}) is the radial solution of (\ref{Hel}) in $%
B_{t}$ that has the matching value on $\partial B_{t}$. In the case where $%
J_{\alpha }(kt)\neq 0$ analogous reasoning yields (\ref{C2}) but with 
\begin{equation*}
a_{k}(t)=\frac{1}{t^{-\alpha }J_{\alpha }(kt)}\left( \frac{-\pi
tJ_{N/2}(kt)Y_{\alpha }(kt)}{2k}+\frac{1}{k^{2}}\right) .
\end{equation*}%
When $J_{\alpha }(kt)\neq 0$ we see from Lemma \ref{Bessel}(iv) that 
\begin{equation*}
a_{k}(t)=\frac{-\pi t^{N/2}}{2k}Y_{N/2}(kt),
\end{equation*}%
and a continuity argument shows that the latter formula remains valid when $%
J_{\alpha }(kt)=0$.
\end{proof}

\begin{lemma}
\label{maxsub} If $u$ and $v$ are $k$-metasubharmonic functions on $\Omega $%
, then so also is $\mathrm{max}\{u,v\}$.
\end{lemma}

\begin{proof}
Kato's inequality for the Laplacian \cite{BP} (or see Corollary 2.3 of \cite%
{GaSj2} for a short alternative proof) tells us that 
\begin{equation*}
(\Delta +k^{2})\mathrm{max}\{u,v\}\geq (\Delta u+k^{2}u)\chi _{\{u\geq
v\}}+(\Delta v+k^{2}v)\chi _{\{u<v\}}.
\end{equation*}
\end{proof}

\bigskip

The corresponding result for the minimum of two $k$-metasuperharmonic
functions clearly also holds.

\begin{lemma}
\label{basic}(a) If $s$ is a$\ k$-metasubharmonic function on a domain $%
\Omega $ and $s\leq 0$, then either $s<0$ or $s\equiv 0$ in $\Omega $.%
\newline
(b) A $k$-metasubharmonic function $s$ cannot have a negative local maximum. 
\newline
(c) (Hopf lemma) Suppose that $u\in C^{1}(\overline{\Omega })$ is
non-negative and $k$-metasuperharmonic in the domain $\Omega $, where $%
\partial \Omega $ is of class $C^{2}$. If there is a boundary point $x$
where $u$ has a local minimum on $\overline{\Omega }$ and $\dfrac{\partial u%
}{\partial n}(x)=0$, then $u\equiv u(x)$ in $\Omega $.
\end{lemma}

\begin{proof}
(a) This follows from the fact that $s$ is subharmonic in the usual sense.

(b) Suppose that $s$ has a negative local maximum at $z$. If we define $%
v=s-s(z)$, then $v$ will be $k$-metasubharmonic and nonpositive on a small
ball $B_{\delta }(z)$. Since $v(z)=0$, it follows from part (a) that $s=s(z)$
on $B_{\delta }(z)$, which leads to the contradictory conclusion that $%
(\Delta +k^{2})s=k^{2}s(z)<0$ there.

(c) Since $u$ is superharmonic on $\Omega $, the standard Hopf lemma applies.
\end{proof}

\bigskip

The first (Dirichlet) eigenvalue of the Laplacian on a bounded domain $%
\Omega $ is usually defined by minimizing the Rayleigh quotient; that is, 
\begin{equation*}
\lambda _{1}^{\ast }(\Omega )=\mathrm{min}\left\{ \frac{\int_{\Omega
}|\nabla h|^{2}\,dm}{\int_{\Omega }h^{2}\,dm}:h\in W_{0}^{1,2}(\Omega
)\backslash \{0\}\right\} .
\end{equation*}%
It follows from the Poincar\'{e}{} inequality (see (7.44) in \cite{GT}) that $\lambda _{1}^{\ast }(\Omega )>0$, and the existence of minimizers in the
class $W_{0}^{1,2}(\Omega )$ can be proved by the direct method from the
calculus of variations (see Section 8.12 in \cite{GT}). Without the
assumption that $\Omega $ is regular for the Dirichlet problem the boundary
value $0$ is only attained in a weak sense, which means that this value is
attained on $\partial \Omega $ apart, possibly, from a polar set. The
corresponding eigenspace consists of the minimizers of the Rayleigh quotient
and $0$, and is spanned by a single positive eigenfunction $h_{\Omega }$
(see Theorem 8.38 of \cite{GT}; the connectedness of $\Omega $ is crucial
here). We will now verify that, when $\Omega $ is connected, $\lambda
_{1}^{\ast }(\Omega )$ coincides with the quantity $\lambda _{1}(\Omega )$
defined in (\ref{lambda1}).

\begin{proposition}
\label{equiv}If $\Omega $ is a bounded domain, then $\lambda _{1}(\Omega
)=\lambda _{1}^{\ast }(\Omega )=\lambda _{1}^{\ast \ast }(\Omega )$, where 
\begin{equation*}
\lambda _{1}^{\ast \ast }(\Omega )=\mathrm{min}\left\{ \lambda :\text{there
is a nonnegative }\sqrt{\lambda }\text{-metasubharmonic function }s\in
W_{0}^{1,2}(\Omega )\backslash \{0\}\right\} .
\end{equation*}
\end{proposition}

\begin{proof}
Since $h_{\Omega }$ is positive and $\lambda _{1}^{\ast }(\Omega )$%
-metaharmonic on $\Omega $, it is clear that $\lambda _{1}(\Omega )\geq
\lambda _{1}^{\ast }(\Omega )\geq \lambda _{1}^{\ast \ast }(\Omega )$.
Further, if $s\in W_{0}^{1,2}(\Omega )\backslash \{0\}$ is nonnegative and $%
\sqrt{\lambda }-$metasubharmonic, then%
\begin{equation*}
\lambda _{1}^{\ast }(\Omega )\leq \frac{\int_{\Omega }|\nabla s|^{2}\,dm}{%
\int_{\Omega }s^{2}\,dm}=\frac{\int_{\Omega }s(-\Delta s)\,dm}{\int_{\Omega
}s^{2}\,dm}\leq \lambda ,
\end{equation*}%
so $\lambda _{1}^{\ast \ast }(\Omega )\geq \lambda _{1}^{\ast }(\Omega )$.
It remains to check that $\lambda _{1}^{\ast }(\Omega )\geq \lambda
_{1}(\Omega )$.

We first note that, if $v$ is a positive superharmonic function on $\Omega $%
, then $v\in W_{0}^{1,2}(\Omega )$ if and only if $v$ is the Green potential 
$G_{\Omega }\mu _{v}$ of a measure $\mu _{v}$ with Green energy 
\begin{equation*}
\int_{\Omega }\left\vert \nabla v\right\vert ^{2}\,dm=\int G_{\Omega }\mu
_{v}d\mu _{v}<\infty .
\end{equation*}%
Now suppose that $u$ is a positive $\sqrt{\lambda }$-metasuperharmonic
function on $\Omega $ for some $\lambda >\lambda _{1}^{\ast }(\Omega )$ and
let $v_{t}=\mathrm{min}\{u,th_{\Omega }\}$ $(t>0)$. Then $v_{t}$ is $\sqrt{%
\lambda _{1}^{\ast }(\Omega )}$-metasuperharmonic, by Lemma \ref{maxsub},
and so is also superharmonic. Since $0\leq v_{t}\leq th_{\Omega }$ we know
that $v_{t}$ is the Green potential on $\Omega $ of a measure $\mu _{v_{t}}$%
, and 
\begin{equation*}
\int G_{\Omega }\mu _{v_{t}}d\mu _{v_{t}}\leq t\int G_{\Omega }\mu
_{h_{\Omega }}d\mu _{v_{t}}=t\int G_{\Omega }\mu _{v_{t}}d\mu _{h_{\Omega
}}\leq t^{2}\int G_{\Omega }\mu _{h_{\Omega }}d\mu _{h_{\Omega }}<\infty ,
\end{equation*}%
so $v_{t}\in W_{0}^{1,2}(\Omega )$. Since 
\begin{equation*}
(\Delta +\lambda )h_{\Omega }=(\Delta +\lambda _{1}^{\ast }(\Omega
))h_{\Omega }+(\lambda -\lambda _{1}^{\ast }(\Omega ))h_{\Omega }>0,
\end{equation*}%
we may choose $t$ large enough so that 
\begin{equation*}
\int_{\Omega }s_{t}\left( (\Delta +\lambda _{1}^{\ast }(\Omega
))s_{t}\right) \,dm>0\text{, \ where \ }s_{t}=th_{\Omega }-v_{t}=(th_{\Omega
}-u)^{+}.
\end{equation*}%
This leads to the contradictory conclusion that 
\begin{equation*}
\frac{\int_{\Omega }\left\vert \nabla s_{t}\right\vert ^{2}\,dm}{%
\int_{\Omega }s_{t}^{2}\,dm}=\frac{\int_{\Omega }s_{t}(-\Delta s_{t})\,dm}{%
\int_{\Omega }s_{t}^{2}\,dm}<\lambda _{1}^{\ast }(\Omega ).
\end{equation*}
\end{proof}

The definition of $\lambda _{1}(\Omega )$ in (\ref{lambda1}) makes sense
even if $\Omega $ is not connected, and this will be important in many
formulations below. The following lemma explains this situation.

\begin{lemma}
\label{evnoncon1} If $\Omega $ is a bounded open set, then 
\begin{equation*}
\lambda _{1}(\Omega )=\min \{\lambda _{1}(D):D\text{ is a component of }%
\Omega \}.
\end{equation*}%
Further, there is a finite number of components $D_{1},D_{2},\ldots ,D_{j}$
of $\Omega $ for which $\lambda _{1}(D_{i})=\lambda _{1}(\Omega )$ ($1\leq
i\leq j$).
\end{lemma}

\begin{proof}
Suppose that $\lambda \leq \lambda _{1}(D)$ for each component $D$ of $%
\Omega $. Then, for each choice of $D$, there is by definition, a positive $%
\sqrt{\lambda }$-metasuperharmonic function $u_{D}$ on $D$. If we define $%
u=u_{D}$ on each component $D$ we get a positive $\sqrt{\lambda }$%
-metasuperharmonic function on $\Omega $, so $\lambda \leq \lambda
_{1}(\Omega )$ by definition. This shows that 
\begin{equation*}
\lambda _{1}(\Omega )\geq \inf \{\lambda _{1}(D):D\text{ is a component of }%
\Omega \}.
\end{equation*}%
The opposite inequality is trivial. That this infimum really is a minimum,
and that there are only finitely many of the components $D_{i}$, can be seen
from the fact that $\lambda _{1}(D)=\lambda _{1}^{\ast }(D)$ for each
component $D$ and the Faber-Krahn inequality.
\end{proof}

\bigskip

We now provide a potential theoretic proof of the $k$-maximum principle
mentioned in Section \ref{Intro}.

\begin{proposition}
\label{MP}Let $\Omega $ be a bounded domain and $h$ be a positive
eigenfunction of $-\Delta $ corresponding to $\lambda _{1}(\Omega )$. The
following assertions are equivalent:\newline
(a) the $k$-maximum principle holds on $\Omega $;\newline
(b) $k^{2}<\lambda _{1}(\Omega )$;\newline
(c) there is a positive $k$-metasuperharmonic function on $\Omega $ which is
not a multiple of $h;$\newline
(d) there is a $k$-metaharmonic function $v$ on $\Omega $ such that $v\geq 1$%
.
\end{proposition}

\begin{proof}
(a)$\Longrightarrow $(b) Since $(\Delta +k^{2})h=(k^{2}-\lambda _{1}(\Omega
))h$, the function $h$ would be a counterexample to the $k$-maximum
principle if $k^{2}\geq \lambda _{1}(\Omega )$.

(b)$\Longrightarrow $(a) Suppose that $k<l$, where $l=\sqrt{\lambda
_{1}(\Omega )}$, and let $s$ be an upper bounded $k$-metasubharmonic
function on $\Omega $ satisfying $\limsup s(x)\leq 0$ on $\partial \Omega $
apart from a polar set. Let $H[W,E]$ denote the harmonic measure of a set $%
E\subset \partial W$ relative to an open set $W$. A special case of Theorem
6.2 in Haliste \cite{Hal} tells us that, if $z\in \Omega $, then 
\begin{equation}
H[\Omega \times (-\tau ,t),\Omega \times \{t\}](z,0)\leq C(\Omega ,z)e^{-lt}%
\text{ \ \ \ }(t\geq 2,\tau \geq 1).  \label{Hest}
\end{equation}%
Since the function $(x,t)\mapsto s(x)e^{kt}$ is subharmonic on $\Omega
\times \mathbb{R}$ and there is a positive upper bound $c$ for $s$ on $%
\Omega $, it follows that%
\begin{equation*}
s(z)\leq ce^{kt}H[\Omega \times (-\tau ,t),\Omega \times
\{t\}](z,0)+ce^{-k\tau }H[\Omega \times (-\tau ,t),\Omega \times \{-\tau
\}](z,0).
\end{equation*}%
We can thus use (\ref{Hest}) and let $t\rightarrow \infty $ to see that 
\begin{equation*}
s(z)\leq ce^{-k\tau }H[\Omega \times (-\tau ,\infty ),\Omega \times \{-\tau
\}](z,0),
\end{equation*}%
and then let $\tau \rightarrow \infty $ to see that $s(z)\leq 0$. Since $z$
was an arbitrary point of $\Omega $, it follows that $s\leq 0$ on $\Omega $.

(c)$\Longrightarrow $(b) By definition $\lambda _{1}(\Omega )\geq k^{2}$.
Now suppose that $\lambda _{1}(\Omega )=k^{2}$, let $u$ be a positive $k$%
-metasuperharmonic function on $\Omega $ which is not a multiple of $h$ and
define $m_{1}=\inf_{\Omega }u/h$. The function $v:=u-m_{1}h$ is then $k$%
-metasuperharmonic on $\Omega $, and $v>0$ by assumption. Let $K$ be a
compact set in $\Omega $ such that $\lambda _{1}(\Omega \backslash
K)>\lambda _{1}(\Omega )$, and let $m_{2}=\inf_{K}v/h$, whence $m_{2}>0$. We
can apply the $k$-maximum principle on $\Omega \backslash K$ to the function 
$m_{2}h-v$ to see that $v\geq m_{2}h$ on $\Omega $. This leads to the
contradictory conclusion that $u\geq (m_{1}+m_{2})h$. Hence $\lambda
_{1}(\Omega )>k^{2}$.

(b)$\Longrightarrow $(d) Suppose that $\lambda _{1}(\Omega )>k^{2}$ and
define $f:\mathbb{R}^{N+1}\rightarrow \mathbb{R}$ by $%
f(x,x_{N+1})=e^{kx_{N+1}}$. It follows from (\ref{Hest}) that 
\begin{equation*}
H[\Omega \times (-\tau ,\infty ),\partial \Omega \times (t,\infty
)](z,0)\leq C(\Omega ,z)e^{-lt}\text{ \ \ \ }(t\geq 2,\tau \geq 1),
\end{equation*}%
so $f$ is integrable with respect to harmonic measure for $\Omega \times
(-\tau ,\infty )$. The Dirichlet solution $H_{f}^{\Omega \times (-\tau
,\infty )}$ is bounded at $(z,0)$ by $C(\Omega ,z)$, so we can let $\tau
\rightarrow \infty $ to see that $H_{f}^{\Omega \times \mathbb{R}}$ exists.
We denote this function by $h_{1}$. Then $h_{1}$ majorizes the subharmonic
function $f$ on $\Omega \times \mathbb{R}$ and clearly $%
h_{1}(x,x_{N+1})=e^{kx_{N+1}}h_{1}(x,0)$. If we now define $v(x)=h_{1}(x,0)$%
, then $v$ is $k$-metaharmonic function on $\Omega $ and $v\geq 1$.

(d)$\Longrightarrow $(c) This is obvious.
\end{proof}

\bigskip

We define, for $\delta \in (0,3R_{k})$, 
\begin{equation}
u^{\delta }=u\ast \phi _{\delta }\text{, \ where }\phi _{\delta }=\left(
c_{k}(\delta /3)\right) ^{-3}\left( \chi _{B_{\delta /3}}\ast \chi
_{B_{\delta /3}}\right) \ast \chi _{B_{\delta /3}}.  \label{udelta}
\end{equation}%
If $u$ is defined on an open set $\Omega $, then $u^{\delta }$ is defined
and $C^{2}$ on the set $\Omega ^{\delta }=\{x\in \Omega :\mathrm{dist}%
(x,\Omega ^{c})>\delta \}$ (see, for example, Theorem 1.14 in \cite{Hel}).
Further, 
\begin{equation*}
-(\Delta +k^{2})u^{\delta }=-(\Delta +k^{2})(u\ast \phi _{\delta
})=(-(\Delta +k^{2})u)\ast \phi _{\delta }\text{ \ on \ }\Omega ^{\delta },
\end{equation*}%
and 
\begin{equation*}
\left( U_{k}^{\mu }\right) ^{\delta }=(\Psi _{k}\ast \mu )\ast \phi _{\delta
}=\Psi _{k}\ast (\mu \ast \phi _{\delta })=U_{k}^{\mu ^{\delta }}.
\end{equation*}%
Now suppose that $u$ is $k$-metasuperharmonic on $\Omega $. It follows that $%
u^{\delta }$ is $k$-metasuperharmonic on $\Omega ^{\delta }$. Further, if $%
\delta \in (0,3R_{k})$, then Proposition \ref{MVP} shows that $u^{\delta
}\leq u$ on $\Omega ^{\delta }$, and that $u^{\varepsilon }\geq u^{\delta }$
on $\Omega ^{\delta }$ if $0<\varepsilon <\delta $. We have thus established
the following result.

\begin{theorem}[Mollification]
\label{molli} If $0<\delta <3R_{k}$ and $u$ is $k$-metasuperharmonic on $%
\Omega $, then $u^{\delta }$ is $C^{2}$ and $k$-metasuperharmonic on $\Omega
^{\delta }$. Further, $u^{\delta }(x)\leq u(x)$ on $\Omega ^{\delta }$ and $%
u^{\delta }\nearrow u$ as $\delta \searrow 0$.
\end{theorem}

Let $f\in C(\partial \Omega )$, where $\Omega $ is a bounded open set and $%
\lambda _{1}(\Omega )>k^{2}$. It follows from Proposition \ref{MP} that we
may integrate the function $(x,x_{N+1})\longmapsto e^{kx_{N+1}}f(x)$ with
respect to harmonic measure for $\Omega \times \mathbb{R}$ to obtain a
harmonic function $H$ there. Since $H(x,x_{N+1})=e^{kx_{N+1}}H(x,0)$, the
function defined by $H_{k,f}^{\Omega }(x)=H(x,0)$ solves the Dirichlet
problem for $\Delta +k^{2}$ on $\Omega $ with boundary data $f$. (The set of
irregular boundary points of $\Omega \times \mathbb{R}$ is of the form $%
E\times \mathbb{R}$, where $E$ is the set of irregular boundary points of $%
\Omega $, by Lemma 2 in \cite{G96}.) This definition of $H_{k,f}^{\Omega }$
makes sense for any bounded Borel measurable function $f$ on $\partial
\Omega $. If $\mu $ is a finite measure on $\Omega $, then the mapping $%
f\mapsto \int H_{k,f}^{\Omega }d\mu $ defines a positive linear functional
on $C(\partial \Omega )$, and so is of the form $f\mapsto \int fd\mu
^{\Omega ^{c}}$ for a unique finite positive Radon measure $\mu ^{\Omega
^{c}}$ on $\partial \Omega $. (This measure depends on $k$.) However, since $%
H_{k,1}^{\Omega }>1$ in $\Omega $, we see that 
\begin{equation*}
\mu ^{\Omega ^{c}}(\partial \Omega )=\int H_{k,1}^{\Omega }d\mu >\mu (\Omega
),
\end{equation*}%
so the map $\mu \mapsto \mu ^{\Omega ^{c}}$ is not measure preserving. We
denote the $k$\textit{-harmonic measure} $\delta _{z}^{\Omega ^{c}}$ by $\nu
_{k,z}^{\Omega }$ and note that $U_{k}^{\delta _{z}}-U_{k}^{\nu
_{k,z}^{\Omega }}\geq 0$ with equality on $\Omega ^{c}$ apart possibly from
a polar subset of $\partial \Omega $.

\begin{theorem}[Poisson Modification]
\label{PosMod} Suppose that $u$ is $k$-metasuperharmonic on $\Omega $ and $%
\overline{B_{r}(z)}\subset \Omega $, where $r<R_{k}$. If we define 
\begin{equation*}
v=\left\{ 
\begin{array}{l}
h\text{ in }B_{r}(z) \\ 
u\text{ in }\Omega \setminus B_{r}(z)%
\end{array}%
\right. ,
\end{equation*}%
where $h$ solves the Dirichlet problem for $\Delta +k^{2}$ on $B_{r}(z)$
with boundary data $u$, then $v\leq u$ and $v$ is $k$-metasuperharmonic in $%
\Omega $.
\end{theorem}

\begin{proof}
This follows from Kato's inequality for the Laplacian, since the $k$-maximum
principle holds on $B_{r}(z)$.
\end{proof}

\begin{theorem}
\label{submin} Suppose that $u$ is $k$-metasuperharmonic in $\Omega $ and $%
\mathcal{G\neq \emptyset }$, where 
\begin{equation*}
\mathcal{G}=\{v:v\text{ is }k\text{-metasubharmonic in }\Omega \text{ and }%
v\leq u\text{ in }\Omega \}.
\end{equation*}%
Then $\mathcal{G}$ contains a largest element, which is $k$-metaharmonic on $%
\Omega $.
\end{theorem}

\begin{proof}
This follows from Lemma \ref{maxsub}, Poisson modification, and the fact
that an increasing sequence of $k$-metaharmonic functions that is locally
bounded above has a $k$-metaharmonic limit.
\end{proof}

\section{Partial Balayage\label{PB}}

Suppose that $\rho \in L^{\infty }({\mathbb{R}}^{N})$ and $c^{-1}\leq \rho
\leq c$ for some $c>0$. Given a (positive) measure $\mu $ with compact
support in ${\mathbb{R}}^{N}$ we define 
\begin{equation*}
\mathcal{F}_{k}^{\rho }(\mu )=\left\{ v\in \mathcal{D}^{\prime }({\mathbb{R}}%
^{N}):-(\Delta +k^{2})v\leq \rho \text{, \ }v\leq U_{k}^{\mu }\text{ and }%
\{v<U_{k}^{\mu }\}\text{ is bounded}\right\} .
\end{equation*}%
(Thus the set $\mathcal{F}_{k}(\mu )$ in Section \ref{Intro} corresponds to
the special case where $\rho =1$.)

Now suppose that $\mathcal{F}_{k}^{\rho }(\mu )\neq \emptyset $. If $u,v\in 
\mathcal{F}_{k}^{\rho }(\mu )$, then it follows from the proof of Lemma \ref%
{maxsub} that $\mathrm{max}\{u,v\}\in \mathcal{F}_{k}^{\rho }(\mu )$.
Standard potential theoretic arguments (cf. Section 3.7 of \cite{AG}) now
show that $\mathcal{F}_{k}^{\rho }(\mu )$ has a largest element, which has
an upper semicontinuous representative. We denote this function by $%
V_{k,\rho }^{\mu }$.

\begin{remark}
In \cite{KLSS} the authors treated the case where $\rho =1$ and $\mu \in
L^{\infty }$, where $\mu $ had variable sign and was supported in a small
ball. If we define $w=v+U_{k}^{\mu ^{-}}$, then 
\begin{equation*}
v\leq U_{k}^{\mu ^{+}-\mu ^{-}}\Leftrightarrow w\leq U_{k}^{\mu ^{+}}\text{
\ and \ }-(\Delta +k^{2})v\leq 1\Leftrightarrow -(\Delta +k^{2})w\leq 1+\mu
^{-}.
\end{equation*}%
Thus we can deal with this case by considering $V_{k,1+\mu ^{-}}^{\mu
^{+}}-U_{k}^{\mu ^{-}}$.
\end{remark}

If $\mathcal{F}_{k}^{\rho }(\mu )\neq \emptyset $, then we define%
\begin{align*}
\omega _{k}^{\rho }(\mu )& =\{V_{k,\rho }^{\mu }<U_{k}^{\mu }\}, \\
\mathcal{B}_{k}^{\rho }(\mu )& =-(\Delta +k^{2})V_{k,\rho }^{\mu }, \\
\Omega _{k}^{\rho }(\mu )& =\underset{\{B_{r}(x):(\rho -\mathcal{B}%
_{k}^{\rho }(\mu ))(B_{r}(x))=0\}}{\cup }B_{r}(x).
\end{align*}%
Thus $\mathcal{B}_{k}^{\rho }(\mu )\leq \rho $, and $\Omega _{k}^{\rho }(\mu
)$ is the largest open set where $\mathcal{B}_{k}^{\rho }(\mu )=\rho $. The
sets $\omega _{k}^{\rho }(\mu )$ and $\Omega _{k}^{\rho }(\mu )$ are bounded
by assumption, and we will see below that $\omega _{k}^{\rho }(\mu )\subset
\Omega _{k}^{\rho }(\mu )$. The measure $\mathcal{B}_{k}^{\rho }(\mu )$ is
called the $k$\textit{-partial balayage of }$\mu $\textit{\ onto }$\rho $.
We note that 
\begin{equation*}
V_{k,\rho }^{\mu }=U_{k}^{\mathcal{B}_{k}^{\rho }(\mu )},
\end{equation*}%
since $V_{k,\rho }^{\mu }=U_{k}^{\mu }$ outside a compact set.

\begin{theorem}[Structure Theorem]
\label{strthm} If $\mathcal{F}_{k}^{\rho }(\mu )\neq \emptyset $, then 
\begin{equation}
\mathcal{B}_{k}^{\rho }(\mu )=\rho |_{\omega _{k}^{\rho }(\mu )}+\mu
|_{\omega _{k}^{\rho }(\mu )^{c}}=\rho |_{\Omega _{k}^{\rho }(\mu )}+\mu
|_{\Omega _{k}^{\rho }(\mu )^{c}}\leq \rho  \label{struct}
\end{equation}%
and $\omega _{k}^{\rho }(\mu )\subset \Omega _{k}^{\rho }(\mu )$. Further,
if $\mu (\omega _{k}^{\rho }(\mu )^{c})=0$, then $m(\Omega _{k}^{\rho }(\mu
)\backslash \omega _{k}^{\rho }(\mu ))=0$.
\end{theorem}

\begin{proof}
A Poisson modification argument, applied in small balls where the maximum
principle holds, shows that $\mathcal{B}_{k}^{\rho }(\mu )=\rho $ in $\omega
_{k}^{\rho }(\mu )$. Let $v=U_{k}^{\mu }-V_{k,\rho }^{\mu }$. Then\ $v\geq 0$%
, with equality in $\omega _{k}^{\rho }(\mu )^{c}$. Since harmonic measure
for $\Delta $ is singular with respect to Lebesgue measure \cite{Bou} (see
also \cite{GaSj}),\ we know that $-\Delta v|_{\omega _{k}^{\rho }(\mu
)^{c}}\perp m$. By Kato's inequality, 
\begin{equation*}
0\geq \left( -\Delta v\right) |_{\omega _{k}^{\rho }(\mu )^{c}}=\left(
-(\Delta +k^{2})v\right) |_{\omega _{k}^{\rho }(\mu )^{c}}=\mu |_{\omega
_{k}^{\rho }(\mu )^{c}}-\mathcal{B}_{k}^{\rho }(\mu )|_{\omega _{k}^{\rho
}(\mu )^{c}},
\end{equation*}%
so $\mu |_{\omega _{k}^{\rho }(\mu )^{c}}\leq \mathcal{B}_{k}^{\rho }(\mu
)\leq \rho $ and hence $\left( -\Delta v\right) |_{\omega _{k}^{\rho }(\mu
)^{c}}=0$. This proves the first equality in (\ref{struct}). The remaining
assertions follow from the definition of $\Omega _{k}^{\rho }(\mu )$.
\end{proof}

\bigskip

We will now investigate the set of measures $\mu $ for which $\mathcal{F}%
_{k}^{\rho }(\mu )\neq \emptyset $. We are grateful to Simon Larson for
pointing out the following result, which simplifies some of our exposition.

\begin{proposition}
\label{mollexist} Suppose that $\mathrm{supp}(\mu )$ is compact and there
exists $r\in (0,R_{k}]$ such that $\mu (B_{r}(x))\leq c_{k}(r)$ for all $%
x\in \mathbb{R}^{N}$. If $\rho \geq 1$, then $\mathcal{F}_{k}^{\rho }(\mu
)\neq \emptyset $ and 
\begin{equation}
\mathrm{supp}(\mathcal{B}_{k}^{\rho }(\mu ))\subset \underset{x\in 
\mathrm{supp}(\mu )}{\cup }\overline{B_{r}(x)}.  \label{suppB}
\end{equation}
\end{proposition}

\begin{proof}
If we look at the mollification 
\begin{equation*}
U_{k}^{\mu }\ast h_{r}=U_{k}^{\mu \ast h_{r}},\text{ \ \ where \ \ }h_{r}=%
\frac{1}{c_{k}(r)}\chi _{B_{r}}\text{,}
\end{equation*}%
then Proposition \ref{MVP} shows that $U_{k}^{\mu \ast h_{r}}\leq U_{k}^{\mu
}$. Hence $U_{k}^{\mu \ast h_{r}}\in \mathcal{F}_{k}^{\rho }(\mu )$ since,
by assumption, 
\begin{equation*}
-(\Delta +k^{2})U_{k}^{\mu \ast h_{r}}(x)=\frac{\mu (B_{r}(x))}{c_{k}(r)}%
\leq 1\leq \rho .
\end{equation*}%
Clearly 
\begin{equation*}
\omega _{k}^{\rho }(\mu )\subset \{U_{k}^{\mu \ast h_{r}}<U_{k}^{\mu
}\}\subset \underset{x\in \mathrm{supp}(\mu )}{\cup }\overline{B_{r}(x)},
\end{equation*}%
so (\ref{suppB})\ follows from the structure formula.
\end{proof}

\begin{lemma}
\label{mon1} If $\mathcal{F}_{k}^{\rho }(\mu )\neq \emptyset $ and $\eta
\leq \mu $, then $\mathcal{F}_{k}^{\rho }(\eta )\neq \emptyset $, and 
\begin{align}
U_{k}^{\mu }-V_{k,\rho }^{\mu }& \geq U_{k}^{\eta }-V_{k,\rho }^{\eta },
\label{mona} \\
\omega _{k}^{\rho }(\eta )& \subset \omega _{k}^{\rho }(\mu ),  \label{monb}
\\
\mathcal{B}_{k}^{\rho }(\eta )& \leq \mathcal{B}_{k}^{\rho }(\mu )\text{,}
\label{monc} \\
\Omega _{k}^{\rho }(\eta )& \subset \Omega _{k}^{\rho }(\mu ).  \label{mond}
\end{align}
\end{lemma}

\begin{proof}
If $v\in \mathcal{F}_{k}^{\rho }(\mu )$, then $v-U_{k}^{\mu -\eta }\in 
\mathcal{F}_{k}^{\rho }(\eta )$. In particular, $V_{k,\rho }^{\mu
}-U_{k}^{\mu -\eta }\leq V_{k,\rho }^{\eta }$, whence (\ref{mona}) holds.
Further, if $V_{k,\rho }^{\eta }(x)<U_{k}^{\eta }(x)$, then $V_{k,\rho
}^{\mu }(x)<U_{k}^{\mu }(x)$, so (\ref{monb}) holds. Finally, (\ref{monc})
is now clear from (\ref{struct}), and (\ref{mond}) follows from (\ref{monc})
and the definition of $\Omega _{k}^{\rho }(\cdot )$.
\end{proof}

\begin{remark}
\label{umu}If $\mathcal{F}_{k}^{\rho }(\mu )\neq \emptyset $, then the
function 
\begin{equation*}
u_{k,\rho }^{\mu }:=U_{k}^{\mu }-V_{k,\rho }^{\mu }
\end{equation*}%
is the smallest element in the collection 
\begin{equation}
\{v:v\geq 0,-(\Delta +k^{2})v\geq \mu -\rho \},  \label{umuset}
\end{equation}%
and the above lemma says that it decreases when $\mu $ decreases.
\end{remark}

\begin{lemma}
\label{monink} If $0\leq l\leq k$ and $\mathcal{F}_{k}^{\rho }(\mu )\neq
\emptyset $, then $\mathcal{F}_{l}^{\rho }(\mu )\neq \emptyset $ and 
\begin{equation}
u_{k,\rho }^{\mu }\geq u_{l,\rho }^{\mu },\text{\ \ \ }\omega _{k}^{\rho
}(\mu )\supset \omega _{l}^{\rho }(\mu ).  \label{depk}
\end{equation}
\end{lemma}

\begin{proof}
Since $u_{k,\rho }^{\mu }$ is the smallest element in the collection (\ref%
{umuset}) and 
\begin{equation*}
-(\Delta +l^{2})u_{k,\rho }^{\mu }=-(\Delta +k^{2})u_{k,\rho }^{\mu
}+(k^{2}-l^{2})u_{k,\rho }^{\mu }\geq \mu -\rho ,
\end{equation*}%
we see that $U_{l}^{\mu }-u_{k,\rho }^{\mu }\in \mathcal{F}_{l}^{\rho }(\mu
) $, and (\ref{depk}) follows.
\end{proof}

\begin{lemma}
Let $\mu =\sum\nolimits_{j=1}^{n}\widetilde{c}_{j}\mu _{j}$, where $\mathcal{%
F}_{k}^{\rho }(\mu _{j})\neq \emptyset $ and $\widetilde{c}_{j}\geq 0$ $%
(j=1,...,n)$, and $\sum\nolimits_{j=1}^{n}\widetilde{c}_{j}\leq 1$. Then $%
\mathcal{F}_{k}^{\rho }(\mu )\neq \emptyset $ and 
\begin{equation*}
\omega _{k}^{\rho }(\mu )\subset \bigcup\limits_{j=1}^{n}\omega _{k}^{\rho
}(\mu _{j}),\ \ \ V_{k,\rho }^{\mu }\geq \sum_{j=1}^{n}\widetilde{c}%
_{j}V_{k,\rho }^{\mu _{j}}.
\end{equation*}
\end{lemma}

\begin{proof}
This follows from the observation that, if $v_{j}\in \mathcal{F}_{k}^{\rho
}(\mu _{j})$ $(j=1,...,n)$, then $\sum\nolimits_{j=1}^{n}\widetilde{c}%
_{j}v_{j}\in \mathcal{F}_{k}^{\rho }(\mu )$.
\end{proof}

\begin{proposition}
\label{stbal} If $\mathcal{F}_{k}^{\rho }(\mu _{1}+\mu _{2})\neq \emptyset $
and $\rho \leq \rho _{1}$, then 
\begin{equation*}
V_{k,\rho }^{\mu _{1}+\mu _{2}}=V_{k,\rho }^{\mathcal{B}_{k}^{\rho _{1}}(\mu
_{1})+\mu _{2}}\text{ \ or, equivalently, \ }\mathcal{B}_{k}^{\rho }(\mu
_{1}+\mu _{2})=\mathcal{B}_{k}^{\rho }(\mathcal{B}_{k}^{\rho _{1}}(\mu
_{1})+\mu _{2}).
\end{equation*}
\end{proposition}

\begin{proof}
Clearly 
\begin{equation*}
V_{k,\rho }^{\mu _{1}+\mu _{2}}-U_{k}^{\mu _{2}}\in \mathcal{F}_{k}^{\rho
_{1}}(\mu _{1}).
\end{equation*}%
Hence 
\begin{equation}
V_{k,\rho }^{\mu _{1}+\mu _{2}}\leq U_{k}^{\mathcal{B}_{k}^{\rho _{1}}(\mu
_{1})+\mu _{2}},\text{ \ and so \ }V_{k,\rho }^{\mu _{1}+\mu _{2}}\in 
\mathcal{F}_{k}^{\rho }(\mathcal{B}_{k}^{\rho _{1}}(\mu _{1})+\mu _{2}).
\label{Vest}
\end{equation}%
The desired equality follows, since%
\begin{equation*}
V_{k,\rho }^{\mathcal{B}_{k}^{\rho _{1}}(\mu _{1})+\mu _{2}}\leq U_{k}^{%
\mathcal{B}_{k}^{\rho _{1}}(\mu _{1})+\mu _{2}}\leq U_{k}^{\mu _{1}+\mu
_{2}},\text{ \ and so \ }V_{k,\rho }^{\mathcal{B}_{k}^{\rho _{1}}(\mu
_{1})+\mu _{2}}\in \mathcal{F}_{k}^{\rho }(\mu _{1}+\mu _{2}).
\end{equation*}
\end{proof}

\begin{lemma}
\label{meset} If $\mathcal{F}_{k}^{\rho }(\mu )\neq \emptyset $ and $D$ is a
component of $\omega _{k}^{\rho }(\mu )$, then $\mu (D)\geq (\mu -\rho
)^{+}(D)>0$.
\end{lemma}

\begin{proof}
If there were a component $D$ of $\omega _{k}^{\rho }(\mu )$ such that $(\mu
-\rho )^{+}(D)=0$, then $\mu \leq \rho $ in $D$ and the function 
\begin{equation*}
V=\left\{ 
\begin{array}{l}
V_{k,\rho }^{\mu }\text{ in }D^{c} \\ 
U_{k}^{\mu }\text{ in }D%
\end{array}%
\right.
\end{equation*}%
would belong to $\mathcal{F}_{k}^{\rho }(\mu )$ by Kato's inequality, since $%
-(\Delta +k^{2})U_{k}^{\mu }\leq \rho $ in $D$ and $V\leq U_{k}^{\mu }$
everywhere. However, this leads to the contradictory conclusion that $V\leq
V_{k,\rho }^{\mu }<U_{k}^{\mu }$ in $D$.
\end{proof}

\begin{theorem}
\label{mpsat1}(i) If $D$ is a component of $\omega _{k}^{\rho }(\mu )$ and $%
\mathcal{F}_{k}^{\rho }(\mu +\eta )\neq \emptyset $ for some $\eta $ such
that $\eta (D)>0$, then $D$ satisfies the $k$-maximum principle. In
particular, if this property holds for every component $D$ of $\omega
_{k}^{\rho }(\mu )$, then $k^{2}<\lambda _{1}(\omega _{k}^{\rho }(\mu ))$. 
\newline
(ii) If $\mathcal{F}_{k}^{\rho }(\mu )\neq \emptyset $, then $k^{2}\leq
\lambda _{1}(\omega _{k}^{\rho }(\mu ))$.
\end{theorem}

\begin{proof}
(i) Let $w=U_{k}^{\eta }+V_{k,\rho }^{\mu }-V_{k,\rho }^{\mu +\eta }$. By
Proposition \ref{stbal}, 
\begin{equation*}
V_{k,\rho }^{\mu +\eta }=V_{k,\rho }^{\mathcal{B}_{k}^{\rho }(\mu )+\eta
}\leq U_{k}^{\mathcal{B}_{k}^{\rho }(\mu )}+U_{k}^{\eta }=V_{k,\rho }^{\mu
}+U_{k}^{\eta },
\end{equation*}%
so $w\geq 0$. Since $-(\Delta +k^{2})w=\eta $\ in\ $\omega _{k}^{\rho }(\mu
) $ and $\eta (D)>0$, Proposition \ref{MP} shows that $D$ satisfies the $k$%
-maximum principle.

(ii) If $0<t<1$ and $\mathcal{F}_{k}^{\rho }(\mu )\neq \emptyset $, then $%
\mathcal{F}_{k}^{\rho }(t\mu )\neq \emptyset $ by Lemma \ref{mon1}, and so $%
\mu $ charges each component of $\omega _{k}^{\rho }(t\mu )$, by Lemma \ref%
{meset}. We can thus apply part (i), with $\eta =(1-t)\mu $, to see that $%
k^{2}<\lambda _{1}(\omega _{k}^{\rho }(t\mu ))$. Lemma \ref{mon1} also shows
that both $\omega _{k}^{\rho }(t\mu )$ and $U_{k}^{t\mu }-V_{k,\rho }^{t\mu
} $ increase with $t$. Since $V_{k,\rho }^{t\mu }\geq V_{k,\rho }^{\mu
}+U_{k}^{t\mu }-U_{k}^{\mu }$, by (\ref{mona}), we now see that $V_{k,\rho
}^{\mu }\leq \lim_{t\rightarrow 1-}V_{k,\rho }^{t\mu }\in \mathcal{F}%
_{k}^{\rho }(\mu )$, whence $\lim_{t\rightarrow 1-}V_{k,\rho }^{t\mu
}=V_{k,\rho }^{\mu }$. Thus $\omega _{k}^{\rho }(t\mu )\nearrow \omega
_{k}^{\rho }(\mu )$ as $t\rightarrow 1-$, and the desired conclusion follows.
\end{proof}

\bigskip

The following lemma is a generalization of a standard uniqueness result for
partial balayage. However, the assumption below that $k^{2}\leq \lambda
_{1}(D)$, which trivially holds when $k=0$, is essential when $k>0$.

\begin{lemma}
\label{unique} Suppose that $u=U_{k}^{\mu |_{D}}-U_{k}^{\rho |_{D}}\geq 0$
in $\mathbb{R}^{N}$, where $D=\{x:u(x)>0\}$ is bounded and $\mathrm{supp}%
(\mu) $ is compact. If $k^{2}\leq \lambda _{1}(D)$ and $\mu \leq \rho \text{
in }\mathbb{R}^{N}\setminus D$, then $\mathcal{F}_{k}^{\rho }(\mu )\neq
\emptyset $, $\omega _{k}^{\rho }(\mu )=D$ and $u=u_{k,\rho }^{\mu }$.
\end{lemma}

\begin{proof}
Since $u\geq 0$ and $-(\Delta +k^{2})u\geq \mu -\rho $ by assumption, $%
U_{k}^{\rho |_{D}}+U_{k}^{\mu |_{D^{c}}}\in \mathcal{F}_{k}^{\rho }(\mu )$,
the function $s=u-u_{k,\rho }^{\mu }$ is nonnegative and $\omega _{k}^{\rho
}(\mu )\subset D$. This implies that 
\begin{equation*}
-(\Delta +k^{2})s=(\mu -\rho )|_{D}-(\mu -\rho )|_{\omega _{k}^{\rho }(\mu
)}=(\mu -\rho )|_{D\setminus \omega _{k}^{\rho }(\mu )}\leq 0.
\end{equation*}%
For each component $U$ of $D$ we know by assumption that $\lambda
_{1}(U)\geq k^{2}$. It follows from Proposition 2.7 that either $s\equiv 0$
on $U$ or $s|_{U}$ is an eigenfunction corresponding to $k^{2}=\lambda
_{1}(U)$. However, the latter is impossible, because it would yield a
nontrivial $k$-metaharmonic function $s$ on $\mathbb{R}^{N}$ with compact
support.
\end{proof}

\begin{remark}
The assumptions above regarding $u$ and $D$ are essentially that $u\geq 0$
and that it satisifies $-(\Delta +k^{2})u=(\mu -\rho )|_{D}$ in $\mathbb{R}%
^{N}$, where $D=\{x:u(x)>0\}$ is bounded. It is however necessary to choose
the appropriate representative $U_{k}^{\mu |_{D}}-U_{k}^{\rho |_{D}}$ of the
distribution $u$ to ensure that $D=\omega _{k}^{\rho }(\mu )$.
\end{remark}

\begin{lemma}
\label{mon4} Let $\mu _{n}$ ($n\geq 1$) and $\mu $ be compactly supported
measures.\newline
(a) If $\mu _{n}\nearrow \mu $, $\mathcal{F}_{k}^{\rho }(\mu _{n})\neq
\emptyset $ for each $n$ and there is a compact set $K$ such that $\omega
_{k}^{\rho }(\mu _{n})\subset K$ for all $n$, then $\mathcal{F}_{k}^{\rho
}(\mu )\neq \emptyset $ and $u_{k,\rho }^{\mu _{n}}\nearrow u_{k,\rho }^{\mu
}$.\newline
(b) If $\mu _{n}\searrow \mu $ and $\mathcal{F}_{k}^{\rho }(\mu _{1})\neq
\emptyset $, then $u_{k,\rho }^{\mu _{n}}\searrow u_{k,\rho }^{\mu }$.
\end{lemma}

\begin{proof}
(a) By Remark \ref{umu} and the Structure Theorem $\left( u_{k,\rho }^{\mu
_{n}}\right) $ is increasing, and 
\begin{equation*}
-(\Delta +k^{2})u_{k,\rho }^{\mu _{n}}\geq \left( \mu _{n}-\rho \right)
\left\vert _{\omega _{k}^{\rho }(\mu _{n})}\right. \rightarrow \left( \mu
-\rho \right) \left\vert _{\cup _{n}\omega _{k}^{\rho }(\mu _{n})}\right. .
\end{equation*}%
Hence $\left( u_{k,\rho }^{\mu _{n}}\right) $ has a nonnegative limit $v$
which satisfies $-(\Delta +k^{2})v\geq \mu -\rho $. It follows that $%
U_{k}^{\mu }-v\in \mathcal{F}_{k}^{\rho }(\mu )$ and $u_{k,\rho }^{\mu }$
exists. Finally, by Remark \ref{umu} again, $u_{k,\rho }^{\mu }\geq
\lim_{n\rightarrow \infty }u_{k,\rho }^{\mu _{n}}=v\geq u_{k,\rho }^{\mu }$,
so $u_{k,\rho }^{\mu _{n}}\nearrow u_{k,\rho }^{\mu }$.

(b) In this case $\left( u_{k,\rho }^{\mu _{n}}\right) $ decreases to a
limit $v$ which satisfies $v\geq u_{k,\rho }^{\mu }$, and $\omega _{k}^{\rho
}(\mu )\subset D$, where $D=\{v>0\}$. Since 
\begin{equation*}
-(\Delta +k^{2})u_{k,\rho }^{\mu _{n}}\geq \mu _{n}-\rho \geq \mu -\rho
\end{equation*}%
for each $n$, we see that also 
\begin{equation*}
-(\Delta +k^{2})v\geq \mu -\rho .
\end{equation*}%
Further, since 
\begin{equation*}
-(\Delta +k^{2})u_{k,\rho }^{\mu _{n}}=\mu _{n}-\rho \text{ \ in }\omega
_{k}^{\rho }(\mu _{n})\text{ \ and \ }D\subset \omega _{k}^{\rho }(\mu _{n})%
\text{\ }(n\in \mathbb{N}),
\end{equation*}%
it follows that $-(\Delta +k^{2})v=\mu -\rho $ in $D$. We can now argue as
in the proof of the Structure Theorem too see that $-(\Delta +k^{2})v=\mu
-\rho $ in $\mathbb{R}^{N}$. Since $D\subset \omega _{k}^{\rho }(\mu _{n})$
for each $n$ it follows from Theorem \ref{mpsat1} that $k^{2}\leq \lambda
_{1}(D)$. Hence Lemma \ref{unique} applies, and we see that indeed $D=\omega
_{k}^{\rho }(\mu )$ and $v=u_{k,\rho }^{\mu }$.
\end{proof}

\begin{remark}
We believe that the assumption that $\omega _{k}^{\rho }(\mu _{n})\subset K$
for some compact set $K$ independent of $n$ in (a) is actually redundant,
but have only been able to prove this for constant $\rho $ (see Theorem \ref%
{geom} below). If $\mathcal{F}_{k}^{\rho }(\mu )\neq \emptyset $, then this
assumption, of course, automatically holds.
\end{remark}

\begin{lemma}
\label{mon2} Suppose that $\rho _{n}\searrow \rho $. If\newline
(i) $\mathcal{F}_{k}^{\rho _{n}}(\mu )\neq \emptyset $ for each $n$, and%
\newline
(ii) there is a compact set $K$ such that $\omega _{k}^{\rho _{n}}(\mu
)\subset K$ for each $n$,\newline
then 
\begin{equation*}
\mathcal{F}_{k}^{\rho }(\mu )\neq \emptyset,\, V_{k,\rho _{n}}^{\mu
}\searrow V_{k,\rho }^{\mu }\text{ and }\omega _{k}^{\rho _{n}}(\mu
)\nearrow \omega_k^{\rho }(\mu ).
\end{equation*}%
Further, if $\mathcal{F}_{k}^{\rho }(\mu )\neq \emptyset $, then conditions
(i) and (ii)\ automatically hold.
\end{lemma}

\begin{proof}
Clearly $\mathcal{F}_{k}^{\rho _{m}}(\mu )\subset \mathcal{F}_{k}^{\rho
_{n}}(\mu )$ when $n\leq m$, so the sequence $(V_{k,\rho _{n}}^{\mu })$
decreases to some function $V$, and $(\omega _{k}^{\rho _{n}}(\mu ))$ is
increasing. Since $-(\Delta +k^{2})V_{k,\rho _{n}}^{\mu }\leq \rho _{n}$
with equality in $\omega _{k}^{\rho _{n}}(\mu )$, it follows that $-(\Delta
+k^{2})V\leq \rho $ with equality in $\cup _{n}\omega _{k}^{\rho _{n}}(\mu )$%
. Further, $V=U_{k}^{\mu }$ outside $\cup _{n}\omega _{k}^{\rho _{n}}(\mu )$%
, which is bounded. Hence $V\in \mathcal{F}_{k}^{\rho }(\mu )$ and $%
V=V_{k,\rho }^{\mu }$, and also $\omega _{k}^{\rho _{n}}(\mu )\nearrow
\omega _{k}^{\rho }(\mu )$.

The final assertion is clear since $\mathcal{F}_{k}^{\rho }(\mu )\subset 
\mathcal{F}_{k}^{\rho _{n}}(\mu )$ and $\omega _{k}^{\rho _{n}}(\mu )\subset
\omega _{k}^{\rho }(\mu )$ for each $n$.
\end{proof}

\begin{lemma}
\label{scale1} If $\mathcal{F}_{k}^{\rho }(\mu )\neq \emptyset $ and $t>1$,
then $\mathcal{F}_{k}^{t\rho }(t\mu )\neq \emptyset $ and $V_{k,t\rho
}^{t\mu }=tV_{k,\rho }^{\mu }.$
\end{lemma}

\begin{proof}
This follows from the observation that $v\in \mathcal{F}_{k}^{\rho }(\mu )$
if and only if $tv\in \mathcal{F}_{k}^{t\rho }(t\mu )$.
\end{proof}

\bigskip

\bigskip

The next lemma is based on standard reasoning, but we give a proof for the
sake of completeness.

\begin{lemma}
\label{liminf}Suppose that $\mu _{n}\rightharpoonup \mu $ and there is a
compact set $K$ such that $\mathrm{supp}(\mu _{n})\subset K$ for each $n$.
Then $U_{k}^{\mu }=\liminf_{n\rightarrow \infty }U_{k}^{\mu _{n}}$ $m$%
-almost everywhere, with equality on $K^{c}$.
\end{lemma}

\begin{proof}
We observe that%
\begin{align*}
U_{k}^{\mu }& =\lim_{j\rightarrow \infty }\int \mathrm{min}\{\Psi
_{k}(x-\cdot ),j\}d\mu (x)=\lim_{j\rightarrow \infty }(\lim_{n\rightarrow
\infty }\int \mathrm{min}\{\Psi _{k}(x-\cdot ),j\}d\mu _{n}(x)) \\
& \leq \liminf_{n\rightarrow \infty }\int \Psi _{k}(x-\cdot )d\mu
_{n}(x)=\liminf_{n\rightarrow \infty }U_{k}^{\mu _{n}},
\end{align*}%
and define 
\begin{equation*}
A_{\varepsilon }=\{x\in {\mathbb{R}}^{N}:U_{k}^{\mu }(x)\leq 1/\varepsilon 
\text{ and }\liminf_{n\rightarrow \infty }(U_{k}^{\mu _{n}}(x)-U_{k}^{\mu
}(x))\geq \varepsilon \}\text{ \ \ }(\varepsilon >0).
\end{equation*}%
Then, for any ball $B$, 
\begin{align*}
\varepsilon m(A_{\varepsilon }\cap B)& \leq \int_{A_{\varepsilon }\cap
B}(\liminf_{n\rightarrow \infty }(U_{k}^{\mu _{n}}(y)-U_{k}^{\mu }(y)))dm(y)
\\
& \leq \liminf_{n\rightarrow \infty }\int_{A_{\varepsilon }\cap
B}(U_{k}^{\mu _{n}}(y)-U_{k}^{\mu }(y))dm(y) \\
& =\liminf_{n\rightarrow \infty }\int_{A_{\varepsilon }\cap B}(\int \Psi
_{k}(x-y)d(\mu _{n}-\mu )(x))dm(y) \\
& =\liminf_{n\rightarrow \infty }\int (\int_{A_{\varepsilon }\cap B}\Psi
_{k}(x-y)dm(y))d(\mu _{n}-\mu )(x)=0.
\end{align*}%
Hence $m(A_{\varepsilon })=0$ for each $\varepsilon >0$ and so $U_{k}^{\mu
}=\liminf_{n\rightarrow \infty }U_{k}^{\mu _{n}}$ \ $m$-almost everywhere.
The equality on $K^{c}$ is obvious.
\end{proof}

\begin{theorem}
\label{wcont}Suppose that $\mu _{n}\rightharpoonup \mu $, where each $%
\mathcal{F}_{k}^{\rho }(\mu _{n})$ is non-empty and there is a compact set $%
K $ such that $\mathrm{supp}(\mu _{n})\subset K$ and $\omega _{k}^{\rho
}(\mu _{n})\subset K$ for each $n$. Then $\mathcal{F}_{k}^{\rho }(\mu )$ is
also non-empty and $V_{k,\rho }^{\mu _{n}}\rightarrow V_{k,\rho }^{\mu }$
uniformly in ${\mathbb{R}}^{N}$.
\end{theorem}

\begin{proof}
By the Structure Theorem the measures $\mathcal{B}_{k}^{\rho }(\mu _{n})$
all have support in $K$ and they have uniformly bounded total mass. Thus, if
we choose any subsequence $\left( \mathcal{B}_{k}^{\rho }(\mu
_{n_{j}})\right) $ that is weak* convergent in $L^{\infty }$ to some $\nu $,
and write $\nu _{j}=\mathcal{B}_{k}^{\rho }(\mu _{n_{j}})$, then $\nu \leq
\rho $, the support of $\nu $ is contained in $K$, and $U_{k}^{\nu
_{j}}\rightarrow U_{k}^{\nu }$ pointwise as $j\rightarrow \infty $. We
recall that there exists $c>0$ such that $c^{-1}\leq \rho \leq c$. Since $%
\nu _{j}\leq \rho $, we see that 
\begin{equation*}
\left\vert U_{k}^{\nu _{j}}(x)-U_{k}^{\nu _{j}}(y)\right\vert \leq
c\int_{K}\left\vert \Psi _{k}(x-z)-\Psi _{k}(y-z)\right\vert dm(z),
\end{equation*}%
which implies that the family $(U_{k}^{\nu _{j}})$ is uniformly
equicontinuous. Since also $\Psi _{k}(x)\rightarrow 0$ as $|x|\rightarrow
\infty $ it follows that 
\begin{equation}
U_{k}^{\nu _{j}}\rightarrow U_{k}^{\nu }\text{ \ uniformly in \ }{\mathbb{R}}%
^{N}.  \label{unif}
\end{equation}%
Our aim is to show that $V_{k,\rho }^{\mu }$ exists and equals $U_{k}^{\nu }$%
. The theorem will then follow since $\left( \mathcal{B}_{k}^{\rho }(\mu
_{n_{j}})\right) $ was an arbitrary weak* convergent subsequence of $\left( 
\mathcal{B}_{k}^{\rho }(\mu _{n})\right) $.

By Lemma \ref{liminf}, 
\begin{equation*}
U_{k}^{\mu }=\liminf_{j\rightarrow \infty }U_{k}^{\mu _{n_{j}}}\geq
\liminf_{j\rightarrow \infty }V_{k,\rho }^{\mu
_{n_{j}}}=\liminf_{j\rightarrow \infty }U_{k}^{\nu _{j}}=U_{k}^{\nu },\text{
\ }m\text{-almost everywhere,}
\end{equation*}%
so $U_{k}^{\nu }\leq U_{k}^{\mu }$ everywhere. Since $U_{k}^{\nu
}=U_{k}^{\mu }$ in $K^{c}$, we now see that $U_{k}^{\nu }\in \mathcal{F}%
_{k}^{\rho }(\mu )$ and 
\begin{equation*}
U_{k}^{\nu }\leq V_{k,\rho }^{\mu }\text{, \ with equality outside }K.
\end{equation*}

Now let 
\begin{equation*}
u_{k,\rho }^{\mu _{n_{j}}}=U_{k}^{\mu _{n_{j}}}-V_{k,\rho }^{\mu
_{n_{j}}}=U_{k}^{\mu _{n_{j}}}-U_{k}^{\nu _{j}},
\end{equation*}%
\begin{equation*}
u=U_{k}^{\mu }-U_{k}^{\nu },
\end{equation*}%
and 
\begin{equation*}
D=\{x:u(x)>0\}.
\end{equation*}%
We need to show that $u=u_{k,\rho }^{\mu }$. Clearly $u\geq u_{k,\rho }^{\mu
}$, so $\omega _{k}^{\rho }(\mu )\subset D$. Let $\tilde{K}$ be a compact
subset of $D$. Then there exists $\varepsilon >0$ such that 
\begin{equation*}
\tilde{K}\subset \{x:u(x)>3\varepsilon \},
\end{equation*}%
and since $U_{k}^{\mu ^{\delta }}\nearrow U_{k}^{\mu }$ as $\delta \searrow
0 $ by Theorem \ref{molli}, there exists $\delta \in (0,3R_{k})$ such that 
\begin{equation*}
\tilde{K}\subset \{x:U_{k}^{\mu ^{\delta }}(x)-U_{k}^{\nu }(x)>2\varepsilon
\}.
\end{equation*}

Further, since $U_{k}^{\nu _{j}}\rightarrow U_{k}^{\nu }$ uniformly by (\ref%
{unif}) and, by similar reasoning, $(U_{k}^{\mu _{n_{j}}^{\delta }})$ also
converges uniformly to $U_{k}^{\mu ^{\delta }}$ (for fixed $\delta \in
(0,3R_{k})$), there exists $j_{0}$ such that 
\begin{equation*}
U_{k}^{\mu _{n_{j}}^{\delta }}-U_{k}^{\nu }>\varepsilon \text{ on }%
\widetilde{K}\text{ for all }j\geq j_{0}
\end{equation*}%
and 
\begin{equation*}
U_{k}^{\nu }-U_{k}^{\nu _{j}}\geq -\varepsilon \text{ for all }j\geq j_{0}.
\end{equation*}%
Therefore 
\begin{align*}
\tilde{K}& \subset \{x:U_{k}^{\mu _{n_{j}}^{\delta }}(x)-U_{k}^{\nu
}(x)>\varepsilon \} \\
& \subset \{x:U_{k}^{\mu _{n_{j}}}(x)-U_{k}^{\nu }(x)>\varepsilon \}\subset
\omega _{k}^{\rho }(\mu _{n_{j}})\text{ \ for all }j\geq j_{0}.
\end{align*}%
Hence, by Theorem \ref{mpsat1}, the interior $\tilde{K}^{\circ }$ satisfies $%
k^{2}\leq \lambda _{1}(\tilde{K}^{\circ })$, and since this holds for any
such $\tilde{K}$ it follows that $k^{2}\leq \lambda _{1}(D)$.

Since $-(\Delta +k^{2})u_{k,\rho }^{\mu _{n_{j}}}=\mu _{n_{j}}-\rho $ in $%
\omega _{k}^{\rho }(\mu _{n_{j}})$ for each $j$, it follows that $-(\Delta
+k^{2})u=\mu -\rho $ in each $\tilde{K}^{\circ }$, and hence in $D$. As in
the proof of the Structure Theorem we see that $-(\Delta +k^{2})u=(\mu -\rho
)|_{D}$ in $\mathbb{R}^{N}$, so Lemma \ref{unique} now shows that $%
u=u_{k,\rho }^{\mu }$ as claimed.
\end{proof}

\begin{remark}
We conjecture that, in Theorem \ref{wcont}, the assumption that $\omega
_{k}^{\rho }(\mu _{n})\subset K$ is redundant; that is, for any given
compact set $K$ we expect that there exists $R$, depending on $K$ and $\rho $%
, such that $\omega _{k}^{\rho }(\mu )\subset B_{R}$ for every $\mu $
supported by $K$ such that $\mathcal{F}_{k}^{\rho }(\mu )\neq \emptyset $.
However, we have only been able to prove this for constant $\rho $ (see
Theorem \ref{geom} below).
\end{remark}

\section{A Hele-Shaw type law and quadrature identities\label{HeleShaw}}

Here we prove a domain evolution result for the Helmholtz equation that is
inspired by the study of Hele-Shaw flow for Laplace's equation (see Section %
\ref{Intro}).

\begin{theorem}
\label{HS2}Let $\Omega $ and $D$ be bounded open sets such that $\overline{%
\Omega }\subset D$ and $\lambda _{1}(D)\geq k^{2}$, and let $\mu _{t}=t\eta
+\rho |_{\Omega }$ $(t>0)$, where $\mathrm{supp}(\eta )\subset \Omega $ and $%
\eta (\widetilde{\Omega })>0$ for each component $\widetilde{\Omega }$ of $%
\Omega $. Then the set 
\begin{equation*}
I=\{t>0:\mathcal{F}_{k}^{\rho }(\mu _{t})\neq \emptyset \}
\end{equation*}%
is an interval, and $T:=\sup I<\infty $. Further, $\Omega \subset \omega
_{k}^{\rho }(\mu _{t})$ for all $t\in I$ and

\begin{itemize}
\item[(a)] $\lambda _{1}(\omega _{k}^{\rho}(\mu _{t}))>k^{2}$ when $0<t<T$;

\item[(b)] $\omega _{k}^{\rho}(\mu _{t})\nearrow \omega _{k}^{\rho}(\mu
_{s}) $ as $t\nearrow s$ for any $s\in I$;

\item[(c)] the following Hele-Shaw law holds: 
\begin{equation*}
\omega _{k}^{\rho }(\mu _{t+\varepsilon })=\omega _{k}^{\rho }\left( \rho
|_{\omega _{k}^{\rho }(\mu _{t})}+\varepsilon \eta ^{\omega _{k}^{\rho }(\mu
_{t})^{c}}\right) \text{\ \ if }0<t<t+\varepsilon \in I.
\end{equation*}
\end{itemize}
\end{theorem}

\begin{proof}
Let $W$ be a smoothly bounded open set such that $\overline{\Omega }\subset
W $ and $\overline{W}\subset D$, and let $u=U_{k}^{\eta }$. Then $\lambda
_{1}(W)>k^{2}$ and we can solve the Dirichlet problem for the operator $%
\Delta +k^{2}$ in $W$. The function%
\begin{equation*}
v=\left\{ 
\begin{array}{cc}
H_{k,u}^{W} & \text{in }W \\ 
u & \text{in }\mathbb{R}^{N}\backslash W%
\end{array}%
\right.
\end{equation*}%
satisfies $v\leq u$ in $W$ with strict inequality in at least one component
of $W$, and $v=U_{k}^{\eta ^{\prime }}$ for some measure $\eta ^{\prime }$
on $\partial W$. Further, Proposition \ref{mollexist} shows that, if $%
\varepsilon $ is sufficiently small, then $\mathrm{supp}\left( \mathcal{B}%
_{k}^{\rho }(\varepsilon \eta ^{\prime })\right) \subset D\backslash 
\overline{\Omega }$. It follows that 
\begin{equation*}
U_{k}^{\rho |_{\Omega }}+V_{k,\rho }^{\varepsilon \eta ^{\prime }}\in 
\mathcal{F}_{k}^{\rho }(\mu _{\varepsilon }),
\end{equation*}%
and so $I\neq \emptyset $. Further, if $t_{1}\in I$, then $(0,t_{1})\subset
I $ by Lemma \ref{mon1}. Thus $I$ is an interval. Assertion (a) is now a
consequence of Theorem \ref{mpsat1}, and it follows from the maximum
principle that $\Omega \subset \omega _{k}^{\rho }(\mu _{t})$ for all $t\in
I $ since $-(\Delta +k^{2})(U_{k}^{\mu _{t}}-V_{k,\rho }^{\mu _{t}})=t\eta $
in $\Omega $.

To prove that $I$ is bounded we note from Lemma \ref{monink} that, since $%
\rho \leq c$, 
\begin{equation*}
\omega _{0}^{c}(\mu _{t})\subset \omega _{0}^{\rho }(\mu _{t})\subset \omega
_{k}^{\rho }(\mu _{t})\text{ \ \ }(t\in I),
\end{equation*}%
which in particular shows that%
\begin{equation*}
\lambda _{1}(\omega _{k}^{\rho }(\mu _{t})\leq \lambda _{1}(\omega
_{0}^{c}(\mu _{t})).
\end{equation*}%
We also know that $\omega _{0}^{c}(\mu _{t})$ exists for all $t\in (0,\infty
)$. Further, we claim that, for any $R>0,$ the ball $B_{R}$ is contained in $%
\omega _{0}^{c}(\mu _{t})$ when $t$ is sufficiently large. To see this, we
note that, if $\eta (B_{\delta /2}(x))>0$, then there is a constant $\alpha
>0$ such that the (classical) balayage of $\alpha ^{-1}\eta |_{B_{\delta
/2}(x)}$ onto $\partial B_{\delta }(x)$ is larger than the harmonic measure $%
\nu _{0,x}^{B_{\delta }(x)^{c}}$. For sufficiently large $t$ it is
furthermore clear that $\omega _{0}^{c}(\alpha t\nu _{0,x}^{B_{\delta
}(x)^{c}})$ is a ball centred at $x$ with volume $\alpha t$ which contains $%
B_{\delta }(x)$, whence $\omega _{0}^{c}(t\eta |_{B_{\delta /2}(x)})\supset
\omega _{0}^{c}(\alpha t\nu _{0,x}^{B_{\delta }(x)^{c}})$. Thus, if we
choose $R>0$ so that $\lambda _{1}(B_{R})<k^{2}$, then $\lambda _{1}(\omega
_{0}^{c}(\mu _{t}))<\lambda _{1}(B_{R})<k^{2}$ when $t$ is sufficiently
large, and so $I$ cannot contain arbitrarily large values of $t$, by Theorem %
\ref{mpsat1}.

Part (b) is an immediate consequence of Lemma \ref{mon4}(a).

To prove (c) we use (\ref{Vest}). Since $\mu _{t}(\Omega ^{c})=0$ and $%
\Omega \subset \omega _{k}^{\rho }(\mu _{t})$, this shows that 
\begin{equation*}
V_{k,\rho }^{\mu _{t+\varepsilon }}\leq U_{k}^{\rho |_{\omega _{k}^{\rho
}(\mu _{t})}+\varepsilon \eta }.
\end{equation*}%
Since $U_{k}^{\eta }=U_{k}^{\eta ^{\omega _{k}^{\rho }(\mu _{t})^{c}}}$ on $%
\omega _{k}^{\rho }(\mu _{t})^{c}$, it follows that 
\begin{equation*}
V_{k,\rho }^{\mu _{t+\varepsilon }}\leq U_{k}^{\rho |_{\omega _{k}^{\rho
}(\mu _{t})}+\varepsilon \eta ^{\omega _{k}^{\rho }(\mu _{t})^{c}}}\text{ \
on \ }{\omega _{k}^{\rho }(\mu _{t})^{c},}
\end{equation*}%
and this inequality also holds in $\omega _{k}^{\rho }(\mu _{t})$, by the
maximum principle. Hence 
\begin{equation*}
V_{k,\rho }^{\mu _{t+\varepsilon }}\leq V_{k,\rho }^{\rho |_{\omega
_{k}^{\rho }(\mu _{t})}+\varepsilon \eta ^{\omega _{k}^{\rho }(\mu
_{t})^{c}}}.
\end{equation*}%
The reverse inequality is trivial, so the stated equality holds.
\end{proof}

\begin{remark}
We conjecture that $T$ always belongs to $I$, but have only been able to
prove this for constant $\rho $: see Remark \ref{closcons} below.
\end{remark}

If $\mu $ is a measure with compact support such that $(\mu -m)^{+}\neq 0$,
then it is reasonable to ask if there is an open set $\Omega $ which is a
quadrature domain with respect to $\mu $ for $k$-metasubharmonic functions.
In the following result, which is a slight generalization of Theorem \ref%
{qd0} to allow for non-constant weights, we assume that $\mu ((\omega
_{k}^{\rho }(\mu ))^{c})=0$. As noted in Section \ref{Intro}, we can arrange
that this holds if we replace $\mu $ by $\mu |_{\omega _{k}^{\rho }(\mu )}$,
in view of the Structure Theorem.

\begin{theorem}
\label{qd2}Suppose that $\mathcal{F}_{k}^{\rho}(\mu )\neq \emptyset $ and $%
\mu (\Omega ^{c})=0$, where $\Omega =\omega _{k}^{\rho}(\mu )$. Then%
\begin{equation*}
\int_{\Omega }s \rho \,dm \geq \int s \,d\mu \ \ \ \text{for any integrable }%
k\text{-metasubharmonic function }s\text{ on }\Omega .
\end{equation*}%
In particular, if $\mathrm{supp}(\mu) \subset \Omega $ and $\rho=1$, then $%
\Omega $ is a quadrature domain with respect to $\mu$ for $k$%
-metasubharmonic functions.
\end{theorem}

Theorem \ref{qd2} follows immediately from the following result, which is
proved in Proposition 7.5 of \cite{KLSS}. As noted in Section \ref{Intro},
we now know a much wider range of circumstances in which the hypothesis that 
$\mathcal{F}_{k}^{\rho }(\mu )\neq \emptyset $ holds.

\begin{proposition}
Suppose that $\overline{\omega }\subset \Omega \subset \mathbb{R}^{N}$,
where $\omega $ and $\Omega $ are open sets and $\omega $ is bounded. Then
the linear span, with positive coefficients, of the set%
\begin{equation*}
\left\{ \pm \partial ^{\beta }\Psi _{k}(x-\cdot )|_{\omega }:x\in \Omega
\backslash \omega ,\left\vert \beta \right\vert \leq 1\right\} \cup \left\{
-\Psi _{k}(x-\cdot )|_{\omega }:x\in \omega \right\}
\end{equation*}%
is dense with respect to the $L^{1}(\omega )$-topology in the space%
\begin{equation*}
\left\{ w\in L^{1}(\omega ):(\Delta +k^{2})w\geq 0\text{ \ in \ }\omega
\right\} .
\end{equation*}
\end{proposition}

\bigskip

We will now note some differences between quadrature domains for $k$%
-metasubharmonic functions ($k>0$) and those for subharmonic functions (the
case where $k=0$). In both cases a given measure $\mu $ with compact support
in $\mathbb{R}^{N}$ must, in some sense, be sufficiently concentrated to
guarantee the existence of a quadrature domain for $\mu $. However, unlike
the case $k=0$, where this is the only restriction, when $k>0$ the measure
must also not have too much mass concentrated locally. We saw this for the
case of a point mass in Proposition \ref{pmb}, and will see it for the case
of a uniformly distributed mass on a ball in Proposition \ref{ball}.

Another fundamental difference between the two cases concerns the issue of
uniqueness. When $k=0$ it is well known that any quadrature domain $\Omega $
for subharmonic functions with respect to $\mu $ must satisfy 
\begin{equation*}
\omega _{0}^{1}(\mu )\subset \Omega \subset \Omega _{0}^{1}(\mu ),
\end{equation*}%
and the sets $\Omega _{0}^{1}(\mu )$, $\omega _{0}^{1}(\mu )$ then differ by
at most a set of Lebesgue measure zero. Thus $\Omega $ is unique up to a set
of Lebesgue measure zero. When $k>0$ this kind of uniqueness fails. One way
to see this is through the existence of bounded null quadrature domains for $%
k$-metasubharmonic functions. (These do not exist when $k=0$, because they
would contradict the maximum principle.) For example, we know from
Proposition \ref{MVP} that $U_{k}^{m|_{B_{R_{k}}(y)}}$ is strictly positive
on $B_{R_{k}}(y)$ and zero elsewhere. Thus, if $\Omega $ is a quadrature
domain for $k$-metasubharmonic functions, then so also is $\Omega \cup
B_{R_{k}}(y)$ whenever $\Omega \cap B_{R_{k}}(y)=\emptyset $.

If there is a quadrature domain $\Omega $ for $k$-metasubharmonic functions
with respect to $\mu $, then by definition $U_{k}^{m|_{\Omega }}\in \mathcal{%
F}_{k}^{1}(\mu )$. Hence $\omega _{k}^{1}(\mu )$ exists and is a subset of $%
\Omega $. If $\mathrm{supp}(\mu )\subset \omega _{k}^{1}(\mu )$, then $%
m(\Omega _{k}^{1}(\mu )\setminus \omega _{k}^{1}(\mu ))=0$ and the
quadrature domains $\Omega $ for $k$-metasubharmonic functions with respect
to $\mu $ such that $\omega _{k}^{1}(\mu )\subset \Omega \subset \Omega
_{k}^{1}(\mu )$ are precisely those for which $\lambda _{1}(\Omega )\geq
k^{2}$, by Lemma \ref{unique}.

\section{Partial balayage when $\protect\rho $ is constant\label{constantrho}%
}

In the case where $\rho $ is constant we will now determine the $k$-partial
balayage of some spherically symmetric measures. We will also discuss the
geometric properties of $\omega _{k}^{1}(\mu )$ for more general measures $%
\mu $.

We begin by considering the $k$-partial balayage of a point mass.

\begin{proposition}
\label{pmb} Let $c>0$. \newline
(a) Then $\mathcal{F}_{k}^{1}(c\delta _{x})\neq \emptyset $ if and only if $%
0<c\leq c_{k}(R_{k})$.\newline
(b) If $0<c\leq c_{k}(R_{k})$, then $\mathcal{B}_{k}^{1}(c\delta
_{x})=m|_{B_{r}(x)}$, where $r$ is the unique value in $(0,R_{k}]$ such that 
$c=c_{k}(r)$.
\end{proposition}

\begin{proof}
The \textquotedblleft if\textquotedblright\ assertion in part (a) follows
from Proposition \ref{mollexist}.

We next prove part (b). If $0<c<c_{k}(R_{k})$, then $r<R_{k}$ and the $k$%
-maximum principle holds in $B_{r}(x)$. It follows from the preceding
paragraph that $\omega _{k}^{1}(c\delta _{x})=B_{r}(x)$, and so $\mathcal{B}%
_{k}^{1}(c\delta _{x})=m|_{B_{r}(x)}$. If $c=c_{k}(R_{k})$, then $r=R_{k}$
and the $k$-maximum principle fails to hold in $B_{r}(x)$. Since $%
B_{R_{k}}(x)\subset \omega _{k}^{1}(c\delta _{x})$ by (\ref{monb}), the
function $u:=V_{k,1}^{c\delta _{x}}-U_{k}^{m|_{B_{R_{k}}(x)}}$ is
nonnegative and $k$-metaharmonic on $B_{r}(x)$. It follows from Proposition %
\ref{MP} that either $u\equiv 0$ there or $u$ is a positive eigenfunction
for $B_{R_{k}}(x)$. The latter is impossible, by Lemma \ref{basic}(c), since 
$0\leq u\leq U_{k}^{c\delta _{x}}-U_{k}^{m|_{B_{R_{k}}(x)}}$ on\ $%
B_{R_{k}}(x)$ and the function $U_{k}^{c\delta
_{x}}-U_{k}^{m|_{B_{R_{k}}(x)}}$ has zero normal derivative at $\partial
B_{R_{k}}(x)$.

The \textquotedblleft only if\textquotedblright\ assertion in part (a) now
follows from Theorem \ref{mpsat1}(i), because $\omega
_{k}^{1}(c_{k}(R_{k})\delta _{x})=B_{R_{k}}(x)$ and this set fails to
satisfy the $k$-maximum principle.
\end{proof}

\bigskip

Next, we consider the $k$-partial balayage of a multiple of Lebesgue measure
on a ball, which we may assume to be centred at $0$. We note that, if $\mu
=cm|_{B_{R}}$, where $c\leq 1$, then clearly $\mathcal{B}_{k}^{1}(\mu )=\mu $%
, so the only interesting case is where $c>1$.

\begin{proposition}
\label{ball}Let $\mu =cm|_{B_{R}}$, where $c>1$. \newline
(a) Then $\mathcal{F}_{k}^{1}(\mu )\neq \emptyset $ if and only if $R<R_{k}$
and $c\leq c_{k}(R_{k})/c_{k}(R)$.\newline
(b) If $R<R_{k}$ and $c\leq c_{k}(R_{k})/c_{k}(R)$, then $\mathcal{B}%
_{k}^{1}(\mu )=m|_{B_{r}}$, where $r$ is the unique value in $(R,R_{k}]$
such that $c_{k}(r)=cc_{k}(R)$.
\end{proposition}

\begin{proof}
Suppose firstly that $R<R_{k}$ and $c\leq c_{k}(R_{k})/c_{k}(R)$. Since $%
c_{k}(\cdot )$ is strictly increasing on $(0,R_{k}]$, there is a unique $r$
in $(R,R_{k}]$ such that $c_{k}(r)=cc_{k}(R)$. By Proposition \ref{MVP}, 
\begin{eqnarray}
U_{k}^{m|_{B_{r}}}(z) &=&\int_{B_{r}}\Psi _{k}(y-z)dm(y)  \notag \\
&\leq &\frac{c_{k}(r)}{c_{k}(R)}\int_{B_{R}}\Psi _{k}(y-z)dm(y)=\frac{%
c_{k}(r)}{cc_{k}(R)}U_{k}^{cm|_{B_{R}(x)}}(z),  \label{ballineq}
\end{eqnarray}%
with equality outside $B_{r}(x)$, so $\mathcal{F}_{k}^{1}(\mu )\neq
\emptyset $. This proves the \textquotedblleft if\textquotedblright\
assertion in part (a).

We next prove part (b). If $R<R_{k}$ and $c<c_{k}(R_{k})/c_{k}(R)$, then $%
r<R_{k}$ and the $k$-maximum principle holds in $B_{r}$. It follows that $%
\omega _{k}^{1}(\mu )=B_{r}$, and so $\mathcal{B}_{k}^{1}(\mu )=m|_{B_{r}}$.
If $c=c_{k}(R_{k})/c_{k}(R)$, then $r=R_{k}$ and we can follow the argument
we used in the proof of Proposition \ref{pmb}(b).

It remains to prove the \textquotedblleft only if\textquotedblright\
assertion in part (a). By rotational symmetry and Lemma \ref{meset}, $%
\mathcal{B}_{k}^{1}(\mu )$ must be of the form $m|_{B_{r}}$ for some $r>R$.
Theorem \ref{mpsat1}(ii) tells us that $r\leq R_{k}$ and part (i) of the
same result then shows that $R<R_{k}$. Further, since equality holds in (\ref%
{ballineq}) outside $B_{r}$, we conclude that $c=c_{k}(r)/c_{k}(R)\leq
c_{k}(R_{k})/c_{k}(R)$.
\end{proof}

\bigskip

Propositions \ref{pmb} and \ref{ball} show how, when $r\leq R_{k}$, the ball 
$B_{r}$ can arise from $k$-partial balayage of a point mass, or of $%
cm|_{B_{R}}$ where $c>1$ and $R<r$. The $k$-partial balayage onto $m$ of a
multiple of surface area measure on a sphere is less straightforward. We
will see below when it is a ball and when it is an annular set of the form%
\begin{equation*}
A(r,R)=B_{R}\setminus \overline{B_{r}}\text{ \ \ \ }(0<r<R).
\end{equation*}%
For each $T>0$ we define%
\begin{equation*}
f_{T}(\xi )=\xi ^{N/2}\left\{ J_{N/2}(k\xi )Y_{\alpha }(kT)-Y_{N/2}(k\xi
)J_{\alpha }(kT)\right\} \text{ \ \ \ }(\xi >0),
\end{equation*}%
and define $f_{T}$ at $0$ by its limit there using Lemma \ref{Bessel}(vi).
Next, for each $\xi >0$, we define 
\begin{equation*}
w_{\xi }(r)=\dfrac{1}{k^{2}}\left\{ 1-\dfrac{\pi k}{2}r^{-\alpha }f_{r}(\xi
)\right\} \text{ \ \ \ }(r>0).
\end{equation*}

\begin{lemma}
Let $T>0$ and $\xi >0$. Then $f_{T}$ has a local maximum at $T$, where $%
f_{T}(T)=2T^{\alpha }/(\pi k)$. Also, $w_{\xi }(\xi )=0$ and $w_{\xi
}^{\prime }(\xi )=0$, and 
\begin{equation}
(\Delta +k^{2})w_{\xi }(\left\vert x\right\vert )=1\text{ \ and \ }(\Delta
+k^{2})\frac{f_{T}^{\prime }(\left\vert x\right\vert )}{\left\vert
x\right\vert ^{\alpha +N/2}}=0\text{\ \ \ }(x\neq 0).  \label{wx}
\end{equation}
\end{lemma}

\begin{proof}
By Lemma \ref{Bessel}(iv), $f_{T}(T)=2T^{\alpha }/(\pi k)$, whence\ $w_{\xi
}(\xi )=0$. Lemma \ref{Bessel} also shows that 
\begin{equation*}
f_{T}^{\prime }(r)=kr^{N/2}\left\{ J_{\alpha }(kr)Y_{\alpha }(kT)-Y_{\alpha
}(kr)J_{\alpha }(kT)\right\} ,
\end{equation*}%
so $f_{T}$ has a local maximum at $T$. Further, 
\begin{equation}
w_{\xi }^{\prime }(r)=\frac{\pi }{2}\xi ^{N/2}r^{-\alpha }\left[
J_{N/2}(k\xi )Y_{N/2}(kr)-Y_{N/2}(k\xi )J_{N/2}(kr)\right] ,  \label{der}
\end{equation}%
so $w_{\xi }^{\prime }(\xi )=0$. Finally, (\ref{wx}) holds because functions
of the form (\ref{radial}) are $k$-metaharmonic.
\end{proof}

\bigskip

Noting that $f_{T}^{\prime }(T)=0$, we define $J_{1,T}<J_{2,T}$ to be the
values on either side of $T$ in the list 
\begin{equation*}
0,\text{\ first zero of }f_{T}^{\prime }\text{, second zero of }%
f_{T}^{\prime }\text{, ...,}
\end{equation*}%
and then define $t_{T}$ be the infimum of those $t\in \mathbb{R}$ for which 
\begin{equation*}
\lambda _{1}\left( \left\{ x\in A(J_{1,T},J_{2,T}):t<f_{T}(\left\vert
x\right\vert )<\frac{2T^{\alpha }}{\pi k}\right\} \right) \geq k^{2}.
\end{equation*}

\begin{proposition}
Suppose that $J_{\alpha }(kT)\neq 0$, where $T>0$.\newline
(a)\ If $T>R_{k}$,\ then for any $t\in \lbrack t_{T},2T^{\alpha }/(\pi k))$
there exist $\xi _{1,t}\in (J_{1,T},T)$ and $\xi _{2,t}\in (T,J_{2,T})$ such
that $f_{T}(\xi _{1,t})=t=f_{T}(\xi _{2,t})$ and $m|_{A(\xi _{1,t},\xi
_{2,t})}$ is the partial balayage onto $m$ of the measure 
\begin{equation*}
\frac{c_{k}(\xi _{2,t})-c_{k}(\xi _{1,t})}{d_{k}(T)}\sigma \left\vert
_{\partial B_{T}}\right. .
\end{equation*}%
Further, $\mathcal{F}_{k}^{1}(c\sigma |_{\partial B_{T}})=\emptyset $ when $%
c>\left( c_{k}(\xi _{2,t_{T}})-c_{k}(\xi _{1,t_{T}})\right) /d_{k}(T)$.%
\newline
(b) If $T<R_{k}$ and $f_{T}(0)<f_{T}(R_{k})$, then the same conclusion holds.%
\newline
(c) If $T<R_{k}$ and $t>f_{T}(0)\geq f_{T}(R_{k})$, then the same conclusion
holds.\newline
(d) If $T<R_{k}$ and $f_{T}(0)\geq t\geq f_{T}(R_{k})$, then there exists $%
\xi _{2,t}\in (T,R_{k}]$ such that $f_{T}(\xi _{2,t})=t$ and $m|_{B_{\xi
_{2,t}}}$ is the partial balayage onto $m$ of the measure 
\begin{equation*}
\frac{c_{k}(\xi _{2,t})}{d_{k}(T)}\sigma \left\vert _{\partial B_{T}}\right.
.
\end{equation*}%
Further, $\xi _{2,t}=R_{k}$ when $t=f_{T}(R_{k})$, and $\mathcal{F}%
_{k}^{1}(c\sigma |_{\partial B_{T}})=\emptyset $ when $%
c>c_{k}(R_{k})/d_{k}(T)$.
\end{proposition}

\begin{proof}
(a) Since $f_{T}$ is strictly increasing on $(J_{1,T},T]$ and strictly
decreasing on $[T,J_{2,T})$\ we must have $t_{T}>\mathrm{max}%
\{f_{T}(J_{1,T}),f_{T}(J_{2,T})\}$, because neither $A(J_{1,T},T)$ nor $%
A(T,J_{2,T})$ satisfy the $k$-maximum principle, by (\ref{wx}). Hence, for
any $t\in \lbrack t_{T},2T^{\alpha }/(\pi k))$, there exist unique values $%
\xi _{1,t}\in (J_{1,T},T)$ and $\xi _{2,t}\in $ $(T,J_{2,T})$ such that $%
f_{T}(\xi _{1,t})=t=f_{T}(\xi _{2,t})$. Let%
\begin{equation*}
w_{t}(x)=\left\{ 
\begin{array}{cc}
w_{\xi _{1,t}}(\left\vert x\right\vert ) & (\xi _{1,t}<\left\vert
x\right\vert \leq T) \\ 
w_{\xi _{2},t}(\left\vert x\right\vert ) & (T<\left\vert x\right\vert <\xi
_{2,t}) \\ 
0 & (\text{elsewhere})%
\end{array}%
\right. .
\end{equation*}%
We know that $w_{t}$ is continuous because $w_{\xi }(\xi )=0$ and 
\begin{equation*}
w_{\xi _{1,t}}(T)=\frac{1}{k^{2}}\left\{ 1-\frac{\pi k}{2}r^{-\alpha
}f_{T}(\xi _{1,t})\right\} =\frac{1}{k^{2}}\left\{ 1-\frac{\pi k}{2}%
r^{-\alpha }f_{T}(\xi _{2,t})\right\} =w_{\xi _{2,t}}(T).
\end{equation*}%
Also, $w_{t}\geq 0$ because $f_{T}\leq f_{T}(T)=2T^{\alpha }/(\pi k)$ on $%
[J_{1,T},J_{2,T}]$.

It follows from (\ref{wx}) and (\ref{der}) that 
\begin{equation*}
(\Delta +k^{2})w_{t}=c\sigma \left\vert _{\partial B_{T}}\right.
+m\left\vert _{A(\xi _{1,t},\xi _{2,t})}\right. ,
\end{equation*}%
where%
\begin{eqnarray*}
c &=&w_{\xi _{2,t}}^{\prime }(T)-w_{\xi _{1,t}}^{\prime }(T) \\
&=&\frac{\pi T^{-\alpha }}{2}\xi _{2,t}^{N/2}\left[ J_{N/2}(k\xi
_{2,t})Y_{N/2}(kT)-Y_{N/2}(k\xi _{2,t})J_{N/2}(kT)\right] \\
&&-\frac{\pi T^{-\alpha }}{2}\xi _{1,t}^{N/2}\left[ J_{N/2}(k\xi
_{1,t})Y_{N/2}(kT)-Y_{N/2}(k\xi _{1,t})J_{N/2}(kT)\right] \\
&=&\frac{\pi T^{-\alpha }}{2}Y_{N/2}(kT)\left[ \xi _{2,t}^{N/2}J_{N/2}(k\xi
_{2,t})-\xi _{1,t}^{N/2}J_{N/2}(k\xi _{1,t})\right] \\
&&-\frac{\pi T^{-\alpha }}{2}J_{N/2}(kT)\left[ \xi _{2,t}^{N/2}Y_{N/2}(k\xi
_{2,t})-\xi _{1,t}^{N/2}Y_{N/2}(k\xi _{1,t})\right] .
\end{eqnarray*}%
Since $f_{T}(\xi _{1,t})=f_{T}(\xi _{2,t})$, we know that 
\begin{eqnarray*}
&&J_{\alpha }(kT)\left[ \xi _{2,t}^{N/2}Y_{N/2}(k\xi _{2,t})-\xi
_{1,t}^{N/2}Y_{N/2}(k\xi _{1,t})\right] \\
&=&Y_{\alpha }(kT)\left[ \xi _{2,t}^{N/2}J_{N/2}(k\xi _{2,t})-\xi
_{1,t}^{N/2}J_{N/2}(k\xi _{1,t})\right] .
\end{eqnarray*}%
Recalling that $J_{\alpha }(kT)\neq 0$, we now see that%
\begin{eqnarray*}
c &=&\frac{\pi T^{-\alpha }}{2J_{\alpha }(kT)}\left[ \xi
_{2,t}^{N/2}J_{N/2}(k\xi _{2,t})-\xi _{1,t}^{N/2}J_{N/2}(k\xi _{1,t})\right]
\\
&&\times \left[ Y_{N/2}(kT)J_{\alpha }(kT)-J_{N/2}(kT)Y_{\alpha }(kT)\right]
\\
&=&\frac{-T^{-N/2}}{kJ_{\alpha }(kT)}\left[ \xi _{2,t}^{N/2}J_{N/2}(k\xi
_{2,t})-\xi _{1,t}^{N/2}J_{N/2}(k\xi _{1,t})\right] =-\frac{c_{k}(\xi
_{2,t})-c_{k}(\xi _{1,t})}{d_{k}(T)},
\end{eqnarray*}%
by Lemma \ref{Bessel}(iv), (\ref{ck}) and (\ref{dk}).

Thus%
\begin{equation*}
(\Delta +k^{2})w_{t}=-\mu _{t}+m\left\vert _{_{A(\xi _{1,t},\xi
_{2,t})}}\right. ,\text{ \ where }\mu _{t}=\frac{c_{k}(\xi _{2,t})-c_{k}(\xi
_{1,t})}{d_{k}(T)}\sigma \left\vert _{\partial B_{T}}\right. ,
\end{equation*}%
and so $\mathcal{F}_{k}^{1}(\mu _{t})\neq \emptyset $. If $t\in
(t_{T},2T^{\alpha }/(\pi k))$, then $A(\xi _{1,t},\xi _{2,t})$ satisfies the 
$k$-maximum principle and it follows that $\mathcal{B}_{k}^{1}(\mu
_{t})=m\left\vert _{A(\xi _{1,t},\xi _{2,t})}\right. $. If $t=t_{T}$, the $k$%
-maximum principle fails to hold in $A(\xi _{1,t},\xi _{2,t})$. The function 
$u:=V_{k,1}^{\mu _{t}}-U_{k}^{m|_{A(\xi _{1,t},\xi _{2,t})}}$ is nonnegative
and $k$-metaharmonic on $A(\xi _{1,t},\xi _{2,t})$, so Proposition \ref{MP}
shows that either $u\equiv 0$ there or $u$ is a positive eigenfunction for $%
A(\xi _{1,t},\xi _{2,t})$. The latter is impossible, since $0\leq u\leq
U_{k}^{\mu _{t}}-U_{k}^{m|_{A(\xi _{1,t},\xi _{2,t})}}$ on\ $A(\xi
_{1,t},\xi _{2,t})$ and this last function has zero normal derivative at $%
\partial B_{\xi _{2,t}}$. Finally, it follows from Theorem \ref{mpsat1} that 
$\mathcal{F}_{k}^{1}(c\sigma |_{\partial B_{T}})=\emptyset $ when $c>\left(
c_{k}(\xi _{2,t_{T}})-c_{k}(\xi _{1,t_{T}})\right) /d_{k}(T)$.

(b)\textbf{\ }Since $T<R_{k}$ we know that $J_{1,T}=0$ and, since $%
f_{T}(0)<f_{T}(R_{k})$, it is again the case that $t_{T}>\mathrm{max}%
\{f_{T}(J_{1,T}),f_{T}(J_{2,T})\}$. The argument can now proceed as for part
(a).

(c) When $T<R_{k}$ and $f_{T}(0)\geq f_{T}(R_{k})$ it is again the case that 
$J_{1,T}=0$, but now $t_{T}=f_{T}(R_{k})$. When $t>f_{T}(0)$ the argument
proceeds as before.

(d) It remains to consider the case where $T<R_{k}$ and $f_{T}(0)\geq t\geq
f_{T}(R_{k})$. There exists $\xi _{2,t}\in (T,R_{k})$ such that $f_{T}(\xi
_{2,t})=t$. We define 
\begin{equation*}
w_{t}(x)=\left\{ 
\begin{array}{cc}
\dfrac{1}{k^{2}}\left\{ 1-\dfrac{\pi k}{2}\left\vert x\right\vert ^{-\alpha }%
\dfrac{J_{\alpha }(k\left\vert x\right\vert )}{J_{\alpha }(kT)}f_{T}(\xi
_{2,t})\right\} & (0<\left\vert x\right\vert \leq T) \\ 
w_{\xi _{2},t}(\left\vert x\right\vert ) & (T<\left\vert x\right\vert <\xi
_{2,t}) \\ 
0 & (\text{elsewhere on }\mathbb{R}^{N}\backslash \{0\})%
\end{array}%
\right.
\end{equation*}%
and note that $w_{t}$ is again continuous. Also, by Lemma \ref{Bessel}(v), 
\begin{equation*}
\lim_{x\rightarrow 0}w_{t}(\left\vert x\right\vert )=\dfrac{1}{k^{2}}\left\{
1-\dfrac{\pi k^{N/2}f_{T}(\xi _{2,t})}{2^{N/2}J_{\alpha }(kT)\Gamma (N/2)}%
\right\} \text{ \ and \ }f_{T}(0)=\frac{J_{\alpha }(kT)}{k^{N/2}}\frac{%
2^{N/2}\Gamma (N/2)}{\pi },
\end{equation*}%
so%
\begin{equation*}
\lim_{x\rightarrow 0}w_{t}(\left\vert x\right\vert )=\dfrac{1}{k^{2}}\left(
1-\frac{f_{T}(\xi _{2,t})}{f_{T}(0)}\right) =\dfrac{1}{k^{2}}\left( 1-\frac{t%
}{f_{T}(0)}\right) \geq 0.
\end{equation*}%
Since $w_{t}(x)$ increases with $\left\vert x\right\vert $ on $A(0,T)$ and
we noted previously that $w_{t}\geq 0$ on $A(T,\xi _{2,t})$, we now see that 
$w_{t}\geq 0$ everywhere. Further,%
\begin{equation*}
(\Delta +k^{2})w_{t}=c\sigma \left\vert _{\partial B_{T}}\right.
+m\left\vert _{B_{\xi _{2,t}}}\right. ,
\end{equation*}%
where%
\begin{eqnarray*}
c &=&w_{\xi _{2,t}}^{\prime }(T)-\frac{\pi T^{-\alpha }}{2}\frac{J_{N/2}(kT)%
}{J_{\alpha }(kT)}f_{T}(\xi _{2,t}) \\
&=&\frac{\pi \xi _{2,t}^{N/2}}{2T^{\alpha }}\left\{ J_{N/2}(k\xi
_{2,t})Y_{N/2}(kT)-\frac{J_{N/2}(kT)}{J_{\alpha }(kT)}J_{N/2}(k\xi
_{2,t})Y_{\alpha }(kT)\right\} \\
&=&\frac{\pi \xi _{2,t}^{N/2}}{2T^{\alpha }}\frac{J_{N/2}(k\xi _{2,t})}{%
J_{\alpha }(kT)}\left\{ J_{\alpha }(kT)Y_{N/2}(kT)-J_{N/2}(kT)Y_{\alpha
}(kT)\right\} \\
&=&-\frac{\xi _{2,t}^{N/2}}{kT^{N/2}}\frac{J_{N/2}(k\xi _{2,t})}{J_{\alpha
}(kT)}=-\frac{c_{k}(\xi _{2,t})}{d_{k}(T)},
\end{eqnarray*}%
by Lemma \ref{Bessel}(iv), (\ref{ck}) and (\ref{dk}). Thus%
\begin{equation*}
(\Delta +k^{2})w_{t}=-\mu _{t}+m\left\vert _{B_{\xi _{2,t}}}\right. ,\text{
\ where }\mu _{t}=\frac{c_{k}(\xi _{2,t})}{d_{k}(T)}\sigma \left\vert
_{\partial B_{T}}\right. ,
\end{equation*}%
and so $\mathcal{F}_{k}^{1}(\mu _{t})\neq \emptyset $. If $t>f_{T}(R_{k})$,
then $\xi _{2,t}<R_{k}$, so $B_{\xi _{2,t}}$ satisfies the $k$-maximum
principle and it follows that $\mathcal{B}_{k}^{1}(\mu _{t})=m\left\vert
_{B_{\xi _{2,t}}}\right. $. If $t=f_{T}(R_{k})$, then $\xi _{2,t}=R_{k}$ and
we can argue as in the proof of Proposition \ref{pmb}(b) to see that $%
\mathcal{B}_{k}^{1}(\mu _{t})=m\left\vert _{B_{\xi _{2,t}}}\right. $.
Finally, it follows from Theorem \ref{mpsat1} that $\mathcal{F}%
_{k}^{1}(c\sigma |_{\partial B_{T}})=\emptyset $ when $%
c>c_{k}(R_{k})/d_{k}(T)$.
\end{proof}

\bigskip

Our final goal of this section is to use the moving plane method to
establish a geometric property of $\omega _{k}^{1}(\mu )$. We begin with a
geometric lemma.

\begin{lemma}
\label{geo} Let $\Omega $ be a bounded open set in ${\mathbb{R}}^{N}$ that
contains $\overline{B}_{\varepsilon }$ for some $\varepsilon >0$. Suppose
further that, if $a+tz\in \Omega $ for some $t>\varepsilon $, where $z\in
\partial B_{1}$ and $a\cdot z=0$, then $a-sz\in \Omega $ for every $s\in
(-t,t-2\varepsilon )$. Then 
\begin{equation*}
B_{R}\subset \Omega \subset B_{R+2\varepsilon },\text{ \ where \ }R=\mathrm{%
min}\{\left\vert x\right\vert :x\in \partial \Omega \}.
\end{equation*}
\end{lemma}

\begin{proof}
We choose $x\in \partial \Omega $ such that $R=|x|$, and let $y\in \Omega
\backslash B_{R}$. (If no such $y$ exists, there is nothing to prove.) In
proving that $|y|\leq R+2\varepsilon $ we may assume that $%
|x-y|>2\varepsilon $, for otherwise the desired inequality is trivial. Let 
\begin{equation*}
z=\frac{y-x}{|y-x|},\quad t=y\cdot z\text{ \ \ and \ }s=-x\cdot z,
\end{equation*}%
whence $y-tz=x+sz=a$, say. Since 
\begin{equation*}
\left\vert a\right\vert ^{2}+t^{2}=\left\vert y\right\vert ^{2}\geq
R^{2}=\left\vert x\right\vert ^{2}=\left\vert a\right\vert ^{2}+s^{2}
\end{equation*}%
and $s+t=|y-x|>2\varepsilon $, it follows that $t>\varepsilon $. Further,
since $y=a+tz\in \Omega $ and $x=a-sz\in \partial \Omega $, we know by
hypothesis that $\left\vert s+2\varepsilon \right\vert \geq t$, and so 
\begin{eqnarray*}
|y|^{2} &=&|a|^{2}+t^{2}\leq |a|^{2}+(s+2\varepsilon
)^{2}=|a|^{2}+s^{2}+4\varepsilon s+4\varepsilon ^{2} \\
&\leq &|x|^{2}+4\varepsilon |x|+4\varepsilon ^{2}=(R+2\varepsilon )^{2}.
\end{eqnarray*}
\end{proof}

\begin{theorem}
\label{geom}If $\mu (B_{\varepsilon }^{c})=0$ for some $\varepsilon >0$, and 
$\mathcal{F}_{k}^{1}(\mu )\neq \emptyset $, then 
\begin{equation*}
\omega _{k}^{1}(\mu )\subset B_{R+2\varepsilon },\text{ \ \ where \ \ }R=%
\mathrm{min}\{|x|:x\in \partial \omega _{k}^{1}(\mu )\}.
\end{equation*}
\end{theorem}

\begin{proof}
Let $\delta >0$ and $\Omega _{\delta }=B_{\varepsilon +\delta }\cup \omega
_{k}^{1}(\mu )$. Suppose that $a+tz\in \Omega _{\delta }$ for some $%
t>\varepsilon +\delta $, where $z\in \partial B_{1}$ and $a\cdot z=0$. Let $%
H $ be a hyperplane with normal vector $z$ such that its complementary
components $H_{-},H_{+}$ satisfy $B_{\varepsilon +\delta }\subset H_{-}$ and 
$a+tz\in H_{+}$, and let $\overline{x}$ denote the image of a point $x$
under reflection in $H$. If $v$ is the smallest element of the collection
defined in (\ref{umuset}) with $\rho =1$, then the function 
\begin{equation*}
w(x)=\left\{ 
\begin{array}{cc}
v(x) & (\text{ }x\in H_{-}) \\ 
\min \{v(x),v(\overline{x})\} & (\text{ }x\in H_{+})%
\end{array}%
\right. ,
\end{equation*}%
which also belongs to that collection, must equal $v$. Since $v(a+tz)>0$, we
see that $v(\overline{a+tz})>0$, and so $\overline{a+tz}\in \Omega $. It
follows that the assumptions of Lemma \ref{geo} are satisfied, and so $%
\omega _{k}^{1}(\mu )\subset \Omega _{\delta }\subset B_{R+2(\varepsilon
+\delta )}$. The result follows on letting $\delta \rightarrow 0$.
\end{proof}

\begin{remark}
\label{closcons} Theorem \ref{geom} shows that the assumptions that $\omega
_{k}^{\rho }(\mu _{n})\subset K$ in Lemma \ref{mon4}, and that $\omega
_{k}^{\rho }(\mu _{n})\subset K$ in Theorem \ref{wcont}, are redundant when $%
\rho $ is constant. Likewise, it follows that $T\in I$ in Theorem \ref{HS2}
above. To see this, let $W=\cup _{t<T}\omega _{k}^{1}(\mu _{t})$. Then $%
\omega _{k}^{1}(\mu _{t})\nearrow W$ as $t\nearrow T$, by (\ref{monb}), and $%
W$ is bounded by Theorem \ref{geom}. Since 
\begin{equation*}
tU_{k}^{\eta }+U_{k}^{m|_{\Omega }}\geq U_{k}^{m|_{\omega _{k}^{1}(\mu
_{t})}}\text{ \ with equality precisely on }\omega _{k}^{1}(\mu _{t})^{c}%
\text{ \ \ \ }(0<t<T),
\end{equation*}%
we see from (\ref{mona}) that 
\begin{equation*}
TU_{k}^{\eta }+U_{k}^{m|_{\Omega }}\geq U_{k}^{m|_{W}}\text{ \ with equality
precisely on }W^{c}.
\end{equation*}%
It follows that $\mathcal{F}_{k}^{1}(\mu _{T})\neq \emptyset $ and $W=\omega
_{k}^{1}(\mu _{T})$.
\end{remark}

\newpage

\noindent{\bf Declarations}
\medskip 

\noindent{\bf Ethical Approval} Not applicable 
\medskip 

\noindent{\bf Funding} Not applicable 

\medskip 

\noindent{\bf Availability of data and materials} Not applicable

\bigskip

\noindent \textit{Stephen J. Gardiner}

\noindent School of Mathematics and Statistics

\noindent University College Dublin

\noindent Dublin 4, Ireland

\noindent stephen.gardiner@ucd.ie

\bigskip

\noindent \textit{Tomas Sj\"{o}din}

\noindent Department of Mathematics

\noindent Link\"{o}ping University

\noindent 581 83, Link\"{o}ping

\noindent Sweden

\noindent tomas.sjodin@liu.se

\end{document}